\documentclass{amsart}

\usepackage{latexsym,amsmath,amsfonts,amscd,amssymb, amsthm}

\theoremstyle{plain}
\newtheorem{theorem}{Theorem}[section]
\newtheorem{lemma}{Lemma}[section]
\newtheorem{proposition}{Proposition}[section]
\newtheorem{corollary}{Corollary}[section]
 
\theoremstyle{definition}
\newtheorem{definition}{Definition}[section]
\newtheorem{example}{Example}[section]

\theoremstyle{remark}
\newtheorem{remark}{Remark}[section]

\title{Hermitian analogues of Hilbert's 17-th problem}

\author{John P. D'Angelo}

\address{Dept. of Mathematics, Univ. of Illinois, 1409 W. Green St., Urbana IL 61801}

\email{jpda@math.uiuc.edu}

\begin{document}

\maketitle

\begin{abstract} We pose and discuss several Hermitian analogues of Hilbert's $17$-th problem.
We survey what is known, offer many explicit examples and some proofs, and give applications to
CR geometry. We prove one new algebraic theorem: a non-negative Hermitian symmetric polynomial
divides a nonzero squared norm if and only if it is a quotient of squared norms. We also discuss
a new example of Putinar-Scheiderer.

\medskip

\noindent
{\bf AMS Classification Numbers}: 12D15, 14P05, 15B57, 32A70, 32H35, 32V15.

\medskip

\end{abstract}

\section{Introduction}

Hilbert's $17$-th problem asked whether a non-negative polynomial in several real variables must
be a sum of squares of rational functions. E. Artin answered the question in the affirmative in 1927, using
the Artin-Schreier theory of real fields. Around 1955 A. Robinson gave another 
proof using model theory. See [PD] and [S] for much more information about Hilbert's 
problem. See [R1] for references to recent work and results on concrete aspects of Hilbert's problem. 
See [He] and [HP] for results and applications in the non-commutative setting.

This present paper
aims to survey and organize various results that might be called Hermitian or complex variable analogues of Hilbert's problem.
We also obtain a striking new result in Theorem 5.3.
The results here and their proofs have a rather different flavor from Hilbert's problem; they are connected for example
with ideas such as mapping problems in several complex variables
and CR geometry, analytic tools such as compact operators and the Bergman projection, and metrics on holomorphic vector bundles.
See [CD1], [CD3], [D2], [D3], [D4], [HP], [Q], [TV], and [V] and their references for additional discussion along the lines of this paper.
Both the real and complex cases involve subtle aspects
of zero-sets and how they are defined. The author modestly
hopes that this paper will encourage people to apply the diverse techniques from the real case to
the Hermitian case and that the techniques from the Hermitian case will be useful in the real case as well.

The complex numbers are not an ordered field, and hence to consider non-negativity we must restrict to real-valued polynomials.
The natural starting point will be {\it Hermitian symmetric} polynomials in several complex variables; there is a one-to-one correspondence
between real-valued polynomials on ${\bf R}^{2n}$ and Hermitian symmetric polynomials on ${\bf C}^n \times {\bf C}^n$.
We begin by clarifying this matter.

Let $\rho$ be a real-valued polynomial on ${\bf R}^{2n}$. Call the real variables $(x,y)$; setting $x = {z + {\overline w} \over 2}$
and $y = { z - {\overline w} \over 2i}$  then determines a polynomial $r$ on ${\bf C}^n \times {\bf C}^n$ defined by

$$ r(z, {\overline w}) = \rho({z+{\overline w} \over 2}, {z - {\overline w} \over 2i}). \eqno (1) $$
Polynomials such as $r$ satisfy the Hermitian symmetry condition

$$ r(z, {\overline w}) = {\overline {r(w, {\overline z})}}. \eqno (2) $$
We say that $r$ is Hermitian symmetric in $n$ variables.

\begin{proposition} Let $r: {\bf C}^n \times {\bf C}^n \to {\bf C}$ be a polynomial in $(z,{\overline w})$. The following statements are equivalent:

\begin{itemize}
\item $r$ is Hermitian symmetric. That is,  (2) holds for all $z,w$.

\item The function $z \to r(z,{\overline z})$ is real-valued.

\item  We can write $r(z,{\overline w}) = \sum_{\alpha, \beta} c_{\alpha \beta} z^\alpha {\overline w}^\beta$ where
the {\it matrix of coefficients} is Hermitian symmetric: $c_{\alpha \beta} = {\overline {c_{\beta \alpha}}}$ for all $\alpha, \beta$.
\end{itemize} 
\end{proposition}

Conversely, given a Hermitian symmetric polynomial $r$, the function $z \to r(z,{\overline z})$ can be regarded as a polynomial
in the real and imaginary parts of $z$. We express the ideas via Hermitian symmetric polynomials, for several compelling reasons: the
role of complex analysis is evident, we can polarize by treating $z$ and ${\overline z}$ as independent variables, and
Hermitian symmetry leads to elegance and simplicity not observed in the real setting.

Proposition 1.1 suggests using Hermitian linear algebra to study real-valued polynomials $z \to r(z, {\overline z})$.
The polynomial is Hermitian symmetric if and only if the matrix $C= (c_{\alpha \beta})$ is Hermitian symmetric. On the other hand,
the condition that $r$ be {\it non-negative} as a function is {\bf not} the same as the non-negativity of the matrix
$C$. Much of our work will be firmly based on clarifying this point.

First we introduce a natural concept. We say that ${\bf s}(r) = (A,B)$
if $C$ has $A$ positive and $B$ negative eigenvalues. We call $(A,B)$ the {\it signature pair} of $r$ and
$A+B$ the {\it rank} of $r$. See [D1] or [D3] for versions and applications of the following basic statement.

\begin{proposition} Let $r:{\bf C}^n \times {\bf C}^n \to {\bf C}$ be a Hermitian symmetric polynomial. Then
${\bf s}(r) = (A,B)$ if and only if there are linearly independent holomorphic polynomials $f_1,..., f_A, g_1,...g_B$ such that
$$r(z, {\overline z}) = \sum _{j=1}^A |f_j(z)|^2 - \sum_{j=1}^B |g_j(z)|^2 = ||f(z)||^2 - ||g(z)||^2. \eqno (3) $$
\end{proposition}
If ${\bf s}(r) = (A,0)$ for some $A$ (including $0$) then we call $r$ a {\it squared norm}.
Squared norms take only non-negative values, but non-negative Hermitian symmetric functions
need not be squared norms. In a moment we will give two simple but instructive examples.
In order to clarify these examples and to state several analogues of Hilbert's problem,
we introduce some of the positivity conditions we will be using. See Section 2 for a detailed discussion of these and several other conditions.

\begin{definition} Positivity classes of Hermitian symmetric polynomials. 
\begin{itemize} 
\item ${\mathcal P}_1 = {\mathcal P}_1(n)$ denotes the set of non-negative Hermitian symmetric polynomials in $n$ variables.
\item ${\mathcal P}_\infty = {\mathcal P}_\infty(n)$ denotes the set of Hermitian symmetric polynomials in $n$ variables that
are squared norms of holomorphic polynomial mappings. Thus $r \in {\mathcal P}_\infty$ if and only if $r= ||h||^2$ for a holomorphic polynomial mapping $h$.
\item ${\mathcal Q} = {\mathcal Q}(n)$ denotes the set of polynomials that are quotients of elements of ${\mathcal P}_\infty$. 
Thus $r = {||F||^2 \over ||G||^2}$ for holomorphic polynomial mappings $F,G$.
\item ${\mathcal Q}' = {\mathcal Q}'(n)$ denotes the set of $r \in {\mathcal P}_1$ 
for which there is an $s \in {\mathcal P}_1$ (not identically $0$)
and a holomorphic polynomial mapping $F$ with $rs = ||F||^2$.
\item ${\rm rad} ({\mathcal P}_\infty)$ denotes the set of $r \in {\mathcal P}_1$ for which there is an integer $N$ such that $r^N \in {\mathcal P}_\infty$. Thus
$r^N = ||h||^2$.
\end{itemize}
\end{definition}

The following inclusions are easy to verify:

$$ {\mathcal P}_\infty \subset {\mathcal Q} \subset {\mathcal Q}' \subset {\mathcal P}_1. \eqno (4)$$

$$ {\mathcal P}_\infty \subset {\rm rad}({\mathcal P}_\infty) \subset {\mathcal Q}' \subset {\mathcal P}_1. \eqno (5) $$

Most of these inclusions are strict. Here are simple but instructive examples. See also Example 2.1.

\begin{example} For $\alpha \in {\bf R}$, for $n=1$ and $z=x+iy$, put
$$ r(z,{\overline z}) = \alpha (z+{\overline z})^2 + |z|^2=  (1+4 \alpha)x^2 + y^2. \eqno (6)$$
The following statements hold: 
\begin{itemize}
\item $r \in {\mathcal P}_1$ if and only if $\alpha \ge {-1 \over 4}$.
\item $ r\in {\mathcal P}_\infty$  if and only if $\alpha = 0$. 
\item $r \in {\mathcal Q}$   if and only if $\alpha = 0$. 
\item $r \in {\mathcal Q}'$   if and only if $\alpha = 0$. 
\end{itemize}
Next, for $\lambda \in {\bf R}$, and for $n=2$, put 
$$ r(z,{\overline z}) = |z_1|^4 + \lambda |z_1 z_2|^2 + |z_2|^4. \eqno (7) $$
The following statements hold: 
\begin{itemize}
\item $r \in {\mathcal P}_1$ if and only if $\lambda\ge -2$.
\item $ r\in {\mathcal P}_\infty$  if and only if $\lambda \ge 0$. 
\item $r \in {\mathcal Q}$   if and only if $\lambda > -2$. 
\item $r \in {\mathcal Q}'$   if and only if $\lambda > -2$. 
\end{itemize}
 \end{example}

Example 1.1 shows that two of the containments in (4) are strict. In Example 2.1 we
will see that both containments in (5) are strict.
In Section 5 we prove a surprising result: 
$$ {\mathcal Q} = {\mathcal Q}'. \eqno (8) $$

In many instructive examples the coefficients depend on parameters.
Let $K$ be a closed subset of ${\bf R}^k$. Suppose for each $\lambda \in 
K$ that $c_{\alpha \beta}(\lambda)$ is a Hermitian symmetric matrix
and that the map $\lambda \to c_{\alpha \beta}(\lambda)$
is continuous. We consider the family of Hermitian symmetric polynomials $r_\lambda$ defined for $\lambda \in K$ by
$$ r_\lambda (z,{\overline w}) = \sum c_{\alpha \beta}(\lambda) z^\alpha {\overline w}^\beta .  $$
Let ${\mathcal S}$ be a set of Hermitian symmetric polynomials. 
We say that ${\mathcal S}$ is {\it closed under limits} if, whenever $r_\lambda \in {\mathcal S}$ and 
${\rm lim}(\lambda) = L$, then $r_L \in {\mathcal S}$. By Example 1.1, ${\mathcal Q}$ is not closed under limits. It
is however closed under sum and product. 

In addition to determining which of the containments are strict,
we would like to provide alternative characterizations of the various sets.
For example, two separate results mentioned in Remark 2.1 each characterize ${\mathcal P}_\infty$. These remarks therefore suggest
the following analogues of Hilbert's problem. We discuss answers to 
Analogue 1 from [V], [D4], [D5]. One new aspect of this paper is the introduction and analysis of ${\mathcal Q}'$
and Analogue 2. Theorem 5.3 states that ${\mathcal Q}(n)= {\mathcal Q}'(n)$ for all $n$ and hence answers Analogue 2.
So far Analogue 3 has no nice answer. See Section 7 for some results. See [DP] and the discussion near Example 3.1 for
more on Analogue 4. 
Below we pose additional questions related to all these analogues.

\medskip

{\bf Analogue 1}. Give tractable necessary and sufficient conditions for a polynomial to lie in ${\mathcal Q}$.

\medskip

{\bf Analogue 2}. Give tractable necessary and sufficient conditions for a polynomial to lie in ${\mathcal Q}'$.

\medskip

{\bf Analogue 3}. Give tractable necessary and sufficient conditions for a polynomial to lie in ${\rm rad}({\mathcal P}_\infty)$.

\medskip

{\bf Analogue 4}. Generalize the discussion to algebraic sets and ideals. For example, if a polynomial is positive on an algebraic set,
must it agree with a squared norm there?

\subsection*{Similarities and differences between the real and complex cases}

We pursue the analogy with Hilbert's problem and discover some significant differences.

In the real case, after putting everything over a common denominator, we can state Artin's theorem as follows.
A real polynomial $r$ is non-negative if and only if there is a polynomial $q$ such that
$q^2 r$ is a sum of squares of polynomials. Thus
$$ q^2 r = \sum p_j^2 = ||p||^2. \eqno (9.1)$$

In the complex case, let $r= {||f||^2 \over ||g||^2}$ be a quotient of 
squared norms. Then we have
$$ ||g||^2 r = ||f||^2. \eqno (9.2) $$ 
In both cases we can regard the denominator as a multiplier to bring us into the good situation of squared norms.
Notice however a difference between (9.1) and (9.2). In (9.1) it suffices for the multiplier
to be the square of a single polynomial. In (9.2), even by allowing the rank of $||g||^2$ to be arbitrarily large,
we still do not get all non-negative Hermitian symmetric $r$. Hence we naturally allow the possibility
$$ s r = ||f||^2, \eqno (9.3) $$
where $s$ is an arbitrary non-negative Hermitian symmetric polynomial. 
We still do not get all non-negative polynomials $r$ in this way.

In both the real and complex cases we naturally seek the minimum number of terms required in the sums on the right-hand sides.
A famous result of Pfister [Pf] says in $n$ real dimensions that $2^n$ terms suffice; 
this result is remarkable for two reasons. First, it is independent of the degree. Second, despite considerable work, 
it is unknown what the sharp bound is.
We sound one warning. For $n\ge 2$ there exist non-negative polynomials in $n$ variables that cannot be written
as sums of squares with $2^n$ terms. This statement does not contradict Pfister's result, which says after multiplication by some $q^2$ that the product
can be written as a sum of squares with at most $2^n$ terms.

In the complex case, when $r$ satisfies (9.3) we seek the minimum 
possible rank of $||f||^2$ and when $r$ satisfies (9.2) we also seek the minimum possible rank for $||g||^2$.
No bounds exist depending on only the dimension and the degree of $r$.
If $n=1$ and $r(z,{\overline z}) = (1+|z|^2)^d$, then the rank of $r$ is $d+1$, which obviously depends on the degree.
One cannot write $r$ (or any non-zero multiple of $r$) as a squared norm with fewer than $d+1$ terms. Hence
the analogue of Pfister's result fails. The warning above suggests that we must consider
the possibility of rank dropping under multiplication, and thus motivates Section 9. 
Proposition 9.1 gives an example of maximal collapse in rank. 

\medskip
{\bf Question 1}. Assume $r \in {\mathcal Q}'$. What is the 
minimum rank of any non-zero squared norm $||f||^2$ divisible by $r$?
\medskip

Consider
the polynomial $r_\lambda$ defined in (7).  When $\lambda > -2$, 
$r_\lambda$ is a quotient ${||f_\lambda||^2 \over ||g_\lambda||^2}$ of squared norms and specific maps
$f_\lambda$ and $g_\lambda$ are known. The ranks of these squared norms both tend to infinity
as $\lambda$ tends to $-2$. When $\lambda = -2$, $r$ is not even in ${\mathcal Q}'$. The reason
is that its zero-set is not contained in any complex algebraic variety of 
positive codimension. See the discussion following Definition 1.2.

See Definition 2.3 for the meaning of bihomogeneous. We will pass back and forth between arbitrary Hermitian polynomials
and bihomogeneous ones. 

The paper [S] discusses situations in the real setting regarding sums of squares 
where one must carefully distinguish between positivity and non-negativity. Zero-sets matter.
Our analogues of Hilbert's problem also involve subtle issues about zero-sets. For example, the main result
in [Q] or [CD1] (see Theorem 3.1) implies the following. If $r$ is bihomogeneous and strictly positive away from the origin,
then $r \in {\mathcal Q}$. On the other hand, put $r(z,{\overline z}) = (|z_1|^2 - |z_2|^2)^2$; thus $\lambda=-2$ in (7).
Then $r$ is bihomogeneous but it vanishes along the wrong kind of set for it to divide a squared norm (except $0$).
 Let ${\bf V}(r)$ denote the zero-set of
$r$. For $r$ to be in ${\mathcal Q}$ or ${\mathcal Q}'$, not only must ${\bf V}(r)$ be a complex variety,
but $r$ must define it correctly. 

\begin{definition} A non-negative Hermitian 
 polynomial $r$ has a {\it properly defined} zero-set if there is a 
holomorphic 
polynomial mapping $h$, $\epsilon > 0$,  and a Hermitian 
 polynomial $s$ such that 
 $s \ge \epsilon > 0$ for which $r= ||h||^2 s$. \end{definition}

\begin{example} For $n=1$ and $\alpha \ge 0$ put $r(z,{\overline z}) = \alpha |z|^2 + (z+{\overline z})^2$. If $\alpha > 0$,
then ${\bf V}(r) = \{0\}$, but $r$ defines $0$ in the wrong way.  When
$\alpha = 0$, ${\bf V}(r)$ is even worse;
 it is the line given by $x=0$.
 In either case, $r$ is not in ${\mathcal Q}'$.
The polynomial $1+r$ is also not in ${\mathcal Q}'$. See Theorem 4.1.
\end{example}

Thus, if $r \in {\mathcal Q}'$, then $r$ has a properly defined zero-set. 
By Example 1.2,  the converse fails, even when $r$  is strictly positive.
We also must be careful because there exist positive polynomials whose infima are zero.
See Example 4.1.

Suppose $r(z,{\overline z})$ is divisible (as a polynomial) by 
$||h(z)||^2$
 for a non-constant holomorphic polynomial mapping $h$. Then ${\bf V}(r)$ 
contains the complex variety defined by $h$.
By Lemma 2.4, $r \in {\mathcal Q}$ if and only if ${r \over ||h||^2} \in 
{\mathcal Q}$. 
The same statement holds with ${\mathcal Q}$ replaced by ${\mathcal Q}'$.
There is no loss in generality if we therefore assume that all such 
factors have been canceled.
The result might still have zeroes and hence cause trouble.

We briefly return to the holomorphic decomposition (3) of a Hermitian symmetric polynomial.
Let $r$ be a bihomogeneous Hermitian symmetric polynomial, and assume $r$ is not identically $0$. 
If $r \in {\mathcal P}_\infty$, then we may write $r=||f||^2$, where the components of $f$ are linearly independent. 
By [D1], $f$ is determined up
to a unitary transformation. We regard this situation as understood.

Suppose next that $r \in {\mathcal P}_1$ but $r$ is not in ${\mathcal P}_\infty$.
We write $r=||f||^2 -||g||^2$ where $g \ne 0$
and the components of $f$ and $g$ form a linearly independent set. 
Consider, for each $\lambda \in {\bf R}$, the family
$r_\lambda$ defined by
$$ r_\lambda (z,{\overline z}) = ||f(z)||^2 - \lambda ||g(z)||^2. \eqno (10) $$
For $\lambda \le 0$ it is obvious that $r_\lambda \in {\mathcal P}_\infty$, and this case is understood.
For $0 \le \lambda \le 1$, $r_\lambda$ defines a homotopy between $||f||^2$ and $r$.

Varolin's solution to Analogue 1, 
although expressed in different language in [V], amounts to saying (after dividing out factors of the form $|h|^2$) that 
$r \in {\mathcal Q}$ if and only if there is a $\lambda > 1$ such that $r_\lambda \in {\mathcal P}_1$.
(Equivalently, if there is a constant $c < 1$ such that $||g||^2 \le c ||f||^2$.)
Varolin works in the bihomogeneous setting and even more generally with Hermitian combinations of sections of certain line bundles
over compact complex manifolds. Note that a homogeneous polynomial may be regarded as a section of a power of the hyperplane bundle over projective space.
We briefly discuss such considerations in Section 8.

Varolin's proof uses the resolution of singularities to reduce to bihomogeneous polynomials in two complex variables.
Dehomogenizing then reduces to the case of
one complex dimension, where the problem was solved in [D4]. We give an improved treatment of that work in Section 4.
The proof in one dimension relies on the result in a nondegenerate situation in two dimensions, and we present that information in Section 3.
Zero-sets also play a crucial part in Varolin's approach. Our definition of properly defined zero
set differs slightly from his concept of {\it basic zeroes}. 

Our result that ${\mathcal Q} = {\mathcal Q}'$ solves Analogue 2.
By contrast, the author knows of no satisfactory answer to Analogue 3. See Section 7 for some information.

Analogue 4 has not yet been fully studied. Theorem 3.3 shows
that a Hermitian polynomial positive on the unit sphere agrees with a squared norm there.
Example 3.1 shows that there exist algebraic strongly pseudoconvex hypersurfaces $X$ 
and Hermitian polynomials $f$ such that $f>0$ on $X$ but $f$ agrees with no squared norm there.
One needs analogues of the ${\mathcal P}_k$ spaces for real polynomial ideals; see [DP] for recent work in this
direction. 

\subsection*{Signature pairs}
We close the introduction by considering signature pairs. 
By (9.2) we see that $r \in {\mathcal Q}$ if and only if there is a multiplier $||g||^2 = q \in {\mathcal P}_\infty$
such that $qr \in {\mathcal P}_\infty$. In particular, suppose ${\bf s}(r)=(A,B)$. Then there is a $q$ with ${\bf s}(q)= (N_1,0)$ such that
${\bf s}(qr) = (N_2,0)$. The integers $N_1$ and $N_2$ are complex variable analogues
of the number of terms required in the sums of squares of polynomials from the real variable setting.

We thus study the behavior of the signature pair under multiplication
in analyzing complex variable analogues of Hilbert's problem.
The papers [DL] and [G] apply results about signature pairs to CR Geometry.
In particular Grundmeier [G] computes
the signature pair for various group-invariant Hermitian symmetric polynomials
that determine invariant holomorphic polynomial mappings from spheres to hyperquadrics.

By definition $R \in {\mathcal P}_\infty$ if and only if its signature pair is $(N,0)$
for some $N$. The following question generalizes both Analogue 2 and 
Question 1.

\medskip
{\bf Question 2}.  Suppose $R$ is Hermitian symmetric and that it factors: 
$R=qr$. What can we say about the relationships among the signature pairs for $q,r, R$?

\medskip
Some partial answers to this question appear in [DL].
There exist  pairs of (quite special)
Hermitian symmetric polynomials with arbitrarily large ranks
whose product has signature pair $(2,0)$ and hence rank $2$. We call this phenomenon {\it collapsing of rank}. 
This collapse is sharp, in the sense that we cannot obtain rank $1$. A similar result fails
for real-analytic Hermitian symmetric functions. 

We also note the following fact, which is applied in [DL].
Given a pair $(A,B)$ with $A+B \ge 2$, there exist Hermitian polynomials $r_1$ and $r_2$ such that
all the entries in the signature pairs $(A_j, B_j)$ are positive, but yet the signature pair of the product is $(A,B)$.

The author thanks David Catlin, Jiri Lebl, Dror Varolin, and Mihai Putinar for various useful discussions over the years 
about this material. The author acknowledges support from NSF grant DMS 07-53978.

\section{positivity conditions}

In this section we define various positivity conditions and introduce notation for them.
The definitions include several notions not discussed in the introduction.
We assume that the dimension $n$ is fixed and do not include it in the notation.

\begin{definition} Positivity conditions for Hermitian symmetric polynomials.
\begin{itemize}

\item 1) $r \in {\mathcal P}_1$ if $r(z,{\overline z}) \ge 0$ for all $z$.

\item 2) For $k \in {\bf N}$, we say that $r \in {\mathcal P}_k$ if,
for every choice of $k$ points $z_1,...,z_k$ in ${\bf C}^n$, the matrix
$ r(z_i, {\overline z}_j) $ is non-negative definite. 

\item 3) $r \in {\mathcal P}_\infty$ if there is a holomorphic polynomial mapping
$f$ such that

$$ r(z,{\overline z}) = ||f(z)||^2. \eqno (11) $$
By Proposition 1.1 and linear algebra, (11) holds if and only if the underlying matrix $C$ of Taylor coefficients of $r$ is of the form $A^* A$.

\item 4)  $r \in {\mathcal Q}$ if $r$ is the quotient of elements of ${\mathcal P}_\infty$. In other words, there are holomorphic polynomial mappings
$F$ and $G$  such that
$$ r(z,{\overline z}) = {||F(z)||^2 \over ||G(z)||^2}. \eqno (12) $$

\item 5) $r \in {\mathcal Q}'$ if $r$ is in ${\mathcal P}_1$ and there is $s \in {\mathcal P}_1$ (not identically $0$)
and $||F||^2 \in {\mathcal P}_\infty$ such that $rs = ||F||^2$.

\item 6) $r \in {\rm rad}({\mathcal P}_\infty)$ if $r\ge 0$ and  there is 
a positive integer $N$ such that $r^N \in {\mathcal P}_\infty$. 

\item 7) $r$ satisfies the global Cauchy-Schwarz inequality
if, for all $z$ 
and $w$, 

$$ r(z,{\overline z})r(w,{\overline w}) \ge |r(z,{\overline w})|^2. \eqno (13) $$
If (13) holds, then $r$ achieves only one sign, and $|r| \in {\mathcal P}_1$.
\item 8) $r \in {\mathcal L}$ if  $r \ge 0$ and ${\rm log}(r)$ 
is plurisubharmonic.
\end{itemize}
\end{definition}

\begin{remark} It is well-known, and proved for example in [AM] and [DV], that $r \in {\mathcal P}_k$ for
all $k$ if and only if $r \in {\mathcal P}_\infty$. Thus
$$ {\mathcal P}_\infty = \cap_{j=1}^\infty {\mathcal P}_j, \eqno(14) $$
and (14) gives an alternative definition of ${\mathcal P}_\infty$. 

We mention a second characterization of squared norms from [HP]. Given $r(z,{\overline z})$, replace each variable $z_j$
by a matrix $Z_j$ (of arbitrary size) and replace ${\overline z}_j$ by the adjoint $Z_j^*$ to obtain $r(Z,Z^*)$.
Then $r \in {\mathcal P}_\infty$ if and only if, for all {\it commuting} $n$-tuples $Z=(Z_1,...,Z_n)$, we have
$r(Z,Z^*)\ge 0$. (All such matrices are non-negative definite.) \end{remark}

\begin{remark} Inequality (13) is a curvature condition; it arises when $r$ is regarded as a (possibly degenerate) 
metric on a holomorphic line bundle.
See [Cal], [CD3], [D2], and especially [V]. \end{remark}

From the definitions we immediately see that ${\mathcal P}_{k+1} \subset {\mathcal P}_k$ for all $k$.
Examples from [DV] show that the classes ${\mathcal P}_k$ are distinct.  On the other hand,
if the degrees of polynomials under consideration are bounded, then there is a $k_0$ such that the classes ${\mathcal P}_k$
are the same for $k \ge k_0$.  We recall a concept from [DV]. 

\begin{definition} (Stability Index) Let $S$ be a subset of ${\mathcal P}_1$.
We define $I(S)$ to be the smallest $k$ for which 

$$ S \cap {\mathcal P}_\infty = S \cap {\mathcal P}_k.$$
If no such $k$ exists we write $I(S) = \infty$. When $I(S)$ is finite we say that
$S$ is {\it stable}. \end{definition}

The stability index is computed in [DV] in several interesting situations. From that work
it follows that sets of Hermitian symmetric polynomials of bounded degree are stable.
In other words, only finitely many of the sets ${\mathcal P}_k$ are distinct if we fix the dimension
and bound the degree.

The author does not know if there are stability criteria for the result of [HP] mentioned in Remark 2.1 or the main result in [He].
It seems however that results in this direction could be quite useful.

\begin{remark} For each subset ${\mathcal P}_k$ there 
is a corresponding sharp version;
we demand that the matrix $R(z_i, {\overline z_j})$ be
positive definite whenever the points are distinct. \end{remark}

Recall that ${\mathcal S}$ is {\it closed under limits} if, whenever $r_\lambda \in {\mathcal S}$ and 
${\rm lim}(\lambda) = L$, then $r_L \in {\mathcal S}$. By Example 1.1, ${\mathcal Q}$ is not closed under limits. It
is however closed under sum and product. For each $k$ it is evident
that ${\mathcal P_k}$ is closed under limits. These sets are also closed under sum and product.
 See Lemma 2.4.

The following example offers some insight into the relationships among the various conditions.

\begin{example} [DV] Consider the family of polynomials $r_\lambda$ given by

$$ r_\lambda(z,{\overline z}) = (|z_1|^2 + |z_2|^2)^4 - \lambda |z_1 z_2|^4. \eqno (15) $$
The following hold:

\begin{itemize}
\item $r_\lambda \in {\mathcal P}_1$ if and only if $\lambda \le 16$.

\item $r_\lambda \in {\mathcal Q}$ if and only if $r_\lambda \in {\mathcal Q}'$ if and only if $\lambda < 16$.

\item $r_\lambda \in {\mathcal L}$  if 
and only if $\lambda \le 12$.

\item $r_\lambda \in {\mathcal P}_2$ if and only if $\lambda \le 8$. 

\item $r_\lambda \in {\rm rad}({\mathcal P}_\infty)$ if and only if $\lambda < 8$.

\item For $k >2$, $r_\lambda \in {\mathcal P}_k$ if and only if $r_\lambda \in {\mathcal P}_\infty$.

\item $r_\lambda \in {\mathcal P}_\infty$ if and only if $\lambda \le 6$.
\end{itemize} \end{example}

\begin{remark} 
The corresponding function in $n$ dimensions is
$$ r_a(z, {\overline z}) = ||z||^{4n} - a \left|\prod z_j\right|^4. $$
Then $ r \in {\mathcal P}_1$ for $a \le n^{2n}$,
the Cauchy-Schwarz inequality (13) fails for $a > {n^{2n} \over 2}$, and
$ r \in {\mathcal P}_\infty$ for $a \le {(2n)! \over 2^n}$.
\end{remark}

A sharp form of inequality (13) arises in the isometric embedding theorem from [CD3].
If $r \ge 0$, then (13) is equivalent to
$r \in {\mathcal P}_2$. If $r \in {\rm rad}({\mathcal P}_\infty)$, then $r$ must satisfy (13). Example 1.2 shows that
the converse fails. On the other hand, if $r$ satisfies an appropriate sharp form of (13), then $r \in {\rm rad}({\mathcal P}_\infty)$.
See Theorem 7.1. For a fixed bound on the degree,
the set of polynomials satisfying (13) can be identified with a closed cone in some Euclidean space.
Every point in the interior of this cone corresponds to an element of ${\rm rad}({\mathcal P}_\infty)$, but only
a proper subset of the boundary points do. Analogue 3) is thus closely related to but distinct from studying ${\mathcal P}_2$.

Next we note some obvious properties of the sets ${\mathcal P}_k$
and their analogues for ${\mathcal Q}$.
We continue by discussing bihomogenization and
the surprisingly useful special case when the underlying matrix of a Hermitian symmetric polynomial
is diagonal.

\begin{lemma} Let $t \to z(t)$ be a holomorphic polynomial ${\bf C}^n$-valued mapping.
Let $z^*r$ denote the pullback mapping $t \to r(z(t),{\overline {z(t)}})$.
The following hold:
\begin{itemize} 
\item $r \in {\mathcal Q}(n)$ implies $z^*r \in {\mathcal Q}(1)$.
\item $r \in {\mathcal Q}'(n)$ implies $z^*r \in {\mathcal Q}'(1)$.
\item For $k \ge 1$ or $k=\infty$, $r \in {\mathcal P}_k(n)$ implies $z^*r \in {\mathcal P}_k(1)$.
\end{itemize}
\end{lemma}
\begin{proof} We omit the proof, as these statements are all easy to check.\end{proof}

Let $r$ be a Hermitian symmetric polynomial of degree $m$ in $z$. Even when $r \ge 0$,
its total degree $2d$ can be any even value in the range $m \le 2d \le 2m$.
For squared norms, however, there is an obvious restriction on $2d$.

\begin{lemma} If $r= ||f||^2 \in {\mathcal P}_\infty$, then the total degree
of $r$ is twice the degree of $r$ in $z$. \end{lemma}
\begin{proof} Write $f = f_0 + ...+f_d$ as its expansion into homogeneous parts. Regard $z$ and ${\overline z}$
as independent variables. Then $r$ is of degree $d$ in $z$. Its terms of highest total degree equal $||f_d||^2$  and hence
the total degree of $r$ is $2d$. \end{proof}

\begin{lemma} Each ${\mathcal P}_k$ is closed under sum and 
under product.
For each $k$ we have ${\mathcal P}_{k+1} \subset {\mathcal P}_k$.
Each ${\mathcal P}_k$ is
closed under limits. \end{lemma}

\begin{proof} These facts follow easily from the part of Definition 2.1 giving ${\mathcal P}_k$. 
The proof of closure under product uses a well-known lemma of Schur:
if $(a_{ij})$ and $(b_{ij})$ are non-negative
definite matrices of the same size, then their Schur product $( a_{ij} b_{ij})$ is
also non-negative definite. See [AM] or [D3]. 
\end{proof}

\begin{lemma} ${\mathcal Q}$ is closed under sums and products but not under limits. 
${\mathcal Q}'$ is closed under products but not under limits.\end{lemma}
\begin{proof} Suppose $r = {||f||^2 \over ||g||^2}$ and $R= {||F||^2 \over ||G||^2}$.
Then 
$$ r + R = { || (f \otimes G) \oplus (g \otimes F)||^2 \over ||g \otimes 
G||^2} \eqno (16) $$

$$ r R = { || f \otimes F ||^2 \over ||g \otimes G||^2}. \eqno (17) $$
For the case of ${\mathcal Q}'$ we assume that $r_j s_j = ||f_j||^2$ for $j=1,2$. Then we have
$$ (r_1r_2) (s_1 s_2) = ||f_1 \otimes f_2||^2. \eqno (18) $$
Formula (7) from Example 1.1 shows that ${\mathcal Q}$ and ${\mathcal Q}'$ are not closed under limits.
\end{proof}

\begin{definition}
A Hermitian symmetric polynomial $r$ is called {\it bihomogeneous} of 
total 
degree $2m$ if, for all $\lambda \in {\bf C}$,
$$ r (\lambda w, {\overline \lambda}{\overline w}) = |\lambda|^{2m} 
r(w,{\overline w}). \eqno (19) $$
\end{definition}

For example, $|z|^{2m}$ is bihomogeneous, but $z^k + {\overline z}^k$ is not.
Let $r$ be a Hermitian symmetric polynomial on ${\bf C}^n$, and assume $r$ is of degree $m$ in $z$. (Its total degree
lies in the interval $[m,2m]$.)
We can bihomogenize $r$ by adding the variable $t=z_{n+1}$ and its conjugate. 
Its bihomogenization $Hr$ is defined for $t \ne 0$ by
$$ (Hr)(z, t, {\overline z}, {\overline t}) = |t|^{2m} r \left({z \over 
t}, {{\overline z} \over {\overline t}} \right) \eqno (20) $$
and by continuity at $t=0$. It is evident that if $w=(z,t)$ and $\lambda \in {\bf C}$,
then (19) holds for $Hr$. We say that $Hr$ is bihomogeneous of total degree $2m$. 

For any $k$, $r \in {\mathcal P}_k(n)$ if and only if $Hr \in {\mathcal P}_k(n+1)$.
Furthermore $r \in {\mathcal Q}(n)$ if and only if $Hr \in {\mathcal Q}(n+1)$.  
Thus we will often work in the bihomogeneous setting.

\section{stabilization in the nondegenerate case}

Let $r$ be a bihomogeneous polynomial that is positive away from zero.
In this section we develop the machinery to prove that $r \in {\mathcal Q}$. We give many applications
in the rest of the paper.

Let $B_n$ denote the unit ball in ${\bf C^n}$.
We denote by $A^2(B_n)$ the Hilbert space of square-integrable holomorphic functions
on the ball; it is a closed subspace of $L^2(B_n)$.  We write $V_m$ for the complex vector space of homogeneous holomorphic
polynomials of degree $m$.  The monomials form
a complete orthogonal system for $A^2(B_n)$ and hence $V_m$ is orthogonal to $V_d$ for $m\ne d$.

The {\it Bergman projection} is the self-adjoint projection $P: L^2(B_n) \to A^2(B_n)$.  
The Bergman kernel function for $B_n$ is the Hermitian symmetric real-analytic function
$B(z, {\overline w})$ defined for $f \in L^2(B_n)$ by the formula

$$ Pf(z) = \int_{B_n} B(z,{\overline w}) f(w) dV(w). $$
It is well-known that

$$ B(z, {\overline w}) = {n! \over \pi^n} {1 \over (1 - \langle z,w \rangle)^{n+1}}. \eqno (21) $$
We will use several facts about $P$ and $B$. In particular we note that

$$ B(z,{\overline w}) = \sum_{j=0}^\infty c_j \langle z, w\rangle^j, \eqno (22) $$
where each $c_j$ is a positive number.

\begin{lemma} Let $M$ be multiplication by a bounded function on $L^2(B_n)$. Then the commutator $[P,M]$ is compact on
$L^2(B_n)$. \end{lemma}
\begin{proof}This fact can be directly checked for the ball. See [CD2] for a general result to the effect
that compactness estimates for the ${\overline \partial}$-Neumann problem (well-known for the ball) imply that such a commutator is also
compact. See [Str] for a simpler proof and considerable additional information about compactness 
for the ${\overline \partial}$-Neumann problem. \end{proof}

Note that a power of
the squared Euclidean norm
is itself a squared norm; $||z||^{2d} = ||H_d||^2$, where $H_d$ is the
$d$-fold symmetric tensor product of the identity map with itself.
Observe also that the components of $H_d$ form a basis for $V_d$.

Let us order in some fashion the multi-indices of degree at most $m$. A Hermitian symmetric polynomial $r$  then can be 
considered as the restriction of the Hermitian form in $N$ variables
$$ \sum _{\alpha, \beta =1}^N 
c_{\alpha \beta} \zeta_\alpha {\overline
\zeta_\beta} \eqno (23) $$
to the Veronese variety given by parametric equations $\zeta_\alpha(z) =
z^{\alpha}$. If $r$ is bihomogeneous of total degree $2m$, then $r$
 determines a Hermitian form on $V_m$ via its underlying matrix of coefficients. 
We will use Hermitian symmetric polynomials as integral kernels
of operators on $A^2(B_n)$. Given such an $r$, we define $T_r$ as follows:

$$ (T_r f)(z) = \int_{B_n} r(z, {\overline w}) f(w) dV(w). \eqno (24)  $$

When $r$ is bihomogeneous of total degree $2m$, $T_r$ annihilates every $V_j$ except $V_m$.
Furthermore we have the following simple lemma.

\begin{lemma} Let $r$ be a bihomogeneous polynomial of total degree $2m$.
Then $r \in {\mathcal P}_\infty$ if and only if $T_r$ is non-negative definite on $V_m$. That is
$\langle T_r f, f \rangle \ge 0$ for all $f \in V_m$. Here

$$ \langle T_r f, f \rangle = \int_{B_n} \int_{B_n} r(z,{\overline w}) 
f(w) {\overline {f(z)}} dV(w) dV(z).  \eqno (25) $$
 \end{lemma}

\begin{theorem} ([Q], [CD1])  Let $r(z,{\overline z}) = \sum c_{\alpha \beta} z^\alpha {\overline z}^\beta$ be a
bihomogeneous Hermitian symmetric polynomial of total degree $2m$. The following are equivalent:

1) $r$ achieves a positive minimum value on the sphere.

2) There is an integer $d$ such that the underlying Hermitian matrix 
for $||z||^{2d} r(z,{\overline z})$ is positive
definite. Thus 
$$||z||^{2d} r(z,{\overline z}) = \sum E_{\mu \nu} z^\mu {\overline z}^\nu \eqno (26)
$$
where $(E_{\mu \nu})$ is positive definite.

3) Let  $R_{m+d}$ be the operator defined by 
the kernel 
$ k_d(z,\zeta) = \langle z, \zeta \rangle ^d r(z, {\overline \zeta})$.
There is an integer $d$ such that $R_{m+d}:V_{m+d} \to V_{m+d}$ is 
a positive operator. 

4) There is an integer $d$ and a holomorphic homogeneous vector-valued 
polynomial $g$ of degree $m+d$ such that ${\bf V}(g) = \{0\}$ and such that
$||z||^{2d} r(z,{\overline z}) = ||g(z)||^2 $.

5) Write $r(z,{\overline z}) = ||P(z)||^2 - ||N(z)||^2 $ for 
holomorphic homogeneous vector-valued 
polynomials $P$ and $N$ of degree $m$. Then there is an integer $d$ and a 
linear transformation  $L$ such that the following are true:

5.1) $I - L^*L $ is positive semi-definite. 

5.2 ) $ H_d \otimes N = L (H_d \otimes P) $

5.3) $ \sqrt{I-L^*L} (H_d \otimes P)$ vanishes only at $0$.
\end{theorem}

\begin{corollary} If $r$ is bihomogeneous and positive on the unit sphere, then $r \in {\mathcal Q}$. \end{corollary}

The main assertion that items 1) and 2) are equivalent
was proved in 1967 by Quillen. Unaware of that result, Catlin and the author,
motivated by trying to prove Theorem 3.3 below, found a different proof. 
Both proofs use  analysis; Quillen uses Gaussian integrals and {\it a priori} inequalities
on all of ${\bf C^n}$, whereas Catlin-D'Angelo use compact operators and the Bergman kernel function
on the unit ball $B_n$.  In both approaches it is crucial
that distinct monomials are orthogonal. Theorem 3.1 can be reinterpreted and generalized by expressing it as a statement
about metrics on holomorphic line bundles. See [CD3], [V], and Section 8.

The minimum integer $d$ is the same in items 2) and 3). On the other hand, the integer $d$
in item 4) could be smaller. For example, if $r(z,{\overline z}) = |z_1|^8 + |z_2|^8$, then item 4) holds for $d=0$,
but we require $d \ge 3$ for $(|z_1|^2 + |z_2|^2)^d r(z,{\overline z})$ to satisfy (26) with $(E_{\mu \nu})$ positive definite.

We include item 5) because its generalization leads to a (somewhat unsatisfying) solution to Analogue 2.
Consider replacing $H_d$ by a general holomorphic 
mapping $B$. Suppose $r \in {\mathcal Q}$
and put $r = {||A||^2 \over ||B||^2}$. Then there is an $L$ 
such $B \otimes N = L (B \otimes P)$, 5.1) holds,  and $A = \sqrt{I-L^*L} (B \otimes P)$.
The analogues of conditions 5.1) and 5.2) give in Proposition 3.1 a necessary and sufficient condition for $r$ to be 
in ${\mathcal Q}$. In Theorem 3.1 we know what
to use for $B$, namely $z^{\otimes d}$ for sufficiently large $d$, whereas Proposition 3.1 provides little concrete information on $B$.

\begin{proposition} (An answer to Analogue 1) Suppose $r=||P||^2 - ||N||^2$.
Then $r \in {\mathcal Q}$ if and only if there is a 
holomorphic polynomial mapping $B$ and a linear mapping $L$ such 
that
\begin{itemize}
\item $I-L^*L$ is non-negative definite.
\item $ B \otimes N = L(B \otimes P)$. \end{itemize}\end{proposition}

Next we mention a special case of Theorem 3.1 which goes back to Polya in 1928 and which has many proofs. See for example [D3], [HLP], [R2], and [S].
Let $R$ be a homogeneous polynomial on ${\bf R}^N$. Let $s(x) = 
\sum_{j=1}^N x_j$, and let $H$ denote the part
of the hyperplane defined by $s(x)=1$ and lying in the first orthant. 
Reznick [R2] obtains bounds on the integer $d$ in Theorem 3.2 in terms of the dimension $N$, the degree $m$ of $r$, and the ratio of the maximum and minimum
of $R$ on $H$. To and Yeung [TY] combine the ideas from [R2] and [CD1] to 
give estimates on $d$ from Theorem 3.1 in terms of similar information.
We emphasize that no bounds involving
 only the dimension and the degree are possible. The following result is the special case
of Theorem 3.1 when $r(z,{\overline z})$ depends on only the variables $|z_1|^2,..., |z_N|^2$.

\begin{theorem} (Polya) Let $R(x)$ be a real homogeneous polynomial on ${\bf 
R}^N$. 
Suppose that $R(x) \ge \epsilon > 0$ on $H$. Then there is an integer $d$ 
such that
the polynomial $s^d R$ has all positive coefficients. \end{theorem}

We state a simple corollary of Theorem 3.1 or Theorem 3.2 (going back to Poincar\'e) that can be proved by high school
mathematics.  The result fails in the real-analytic or smooth settings. See [D3] and [HLP] for more information. 
\begin{corollary} 
Let $p$ be a polynomial in one real variable. Then
$p(t) > 0$ for all $t \ge 0$ if and only if
there is an integer $d$ such that the polynomial given by
$(1+t)^d p(t)$ has only 
positive coefficients. The minimum such $d$ can be arbitrarily
large for polynomials of fixed degree. \end{corollary}

See Section 9 for another circumstance where we gain insight
into the general Hermitian case  by 
considering
real polynomials depending on only the variables $|z_1|^2,..., 
|z_n|^2$. We 
close this section by sketching the proof of Theorem 3.1.

\begin{proof} The equivalence of items 2) and 3) follows from Lemma 3.2. Either implies item 4), 
which implies that $r$ is positive away from the origin,
and hence implies item 1). We discuss item 5) later. The crux of the matter is to prove that item 1) implies item 3). 

We want to find an integer $d$ such that $\langle z,w\rangle^d r(z,{\overline w})$ is the integral kernel of a positive operator.
In order to place all these operators 
on the same footing, we study the operator $PM_{r(z,{\overline 
w})}$ with integral kernel equal to 

$$ B(z,{\overline w}) r(z,{\overline w}) = \sum_{j=0}^\infty c_j \langle z, {\overline w} \rangle ^j r(z,{\overline w}). \eqno (27) $$
Recall that each $c_j$ is a positive number. Let $\chi= \chi(w)$ be a non-negative smooth function which is positive at $0$
and has compact support in $B_n$. Consider the operator $M_{r(z,{\overline z})} P + P M_\chi$ with integral kernel
$$ \left( r(z,{\overline z})+ \chi(w) \right)B(z,{\overline w}). \eqno (28) $$
We add and subtract to obtain 
$$ B(z,{\overline w}) r(z,{\overline w})= $$
$$ B(z,{\overline w}) \left( r(z,{\overline w}) - r(z,{\overline z})\right) + 
B(z,{\overline w}) \left( r(z,{\overline z}) + \chi(w) \right) - B(z,{\overline w})\chi(w). \eqno (29) $$
The three terms in (29) define the integral kernels of operators $S_1$, $S_2$, and $S_3$.
The operator $S_3$ is compact on all of $L^2(B_n)$. The operator $S_2$ is easily seen to be positive on $A^2(B_n)$.
The operator $S_1$ can be written as

$$ \sum_{a,b} c_{ab} M_{z^a} [P,M_{{\overline z}^b}]. \eqno (30) $$

We claim that the operator in (30) is also compact; it is a finite sum of bounded operators times commutators of $P$ with bounded operators.
Such commutators are compact by Lemma 3.1. The composition of a bounded
operator with a compact operator is compact, and a finite sum of compact operators is compact.
Hence $S_1$ is compact. It follows that the operator $PM_{r(z,{\overline w})}$
is the sum of a compact operator and a positive operator. Hence, outside of a finite-dimensional subspace,
this operator is itself positive. In other words, for $d$ sufficiently large,
$c_d \langle z, w\rangle^d r(z,{\overline w})$ is the kernel of a positive operator.
Since $c_d > 0$, item 3) follows.

It remains only to check that item 5) is equivalent to the other statements. Assume that
$r = ||P||^2 - ||N||^2$. Then we obtain 
$$ ||H_d||^2 r = ||H_d \otimes P||^2 - ||H_d \otimes N||^2. \eqno(31) $$
If 5.1) and 5.2) hold,
then the right hand side of (31) becomes
$$ ||{\sqrt{I - L^*L}}(H_d \otimes P)||^2= ||g||^2. \eqno (32)$$
If 5.3) also holds, then we obtain 4), and hence item 5) implies item 4).  Conversely suppose that item 4) holds.
Then the right-hand side of (31) is a squared norm $||g||^2$.  We obtain
$$ ||H_d \otimes P||^2 = ||g||^2 + ||H_d \otimes N||^2 = || g \oplus (H_d \otimes N)||^2. \eqno (33) $$
By [D1] there is a unitary map $U$ such that
$$ U((H_d \otimes P) \oplus 0) = (H_d \otimes N) \oplus g. $$
Letting $L$ be one of the blocks of $U$ gives 5.2), and 5.1) follows because $U$ is unitary. The assumption that ${\bf V}(g)= \{0\}$ gives 5.3).
\end{proof}

This decisive theorem has several useful consequences. We pause to prove one such result; others appear in the next two sections.

\begin{theorem} 
Suppose that $r(z,{\overline z})$ is a polynomial that is positive on the unit sphere $S^{2n-1}$.
Then $r$ agrees with the squared norm of a holomorphic polynomial mapping on $S^{2n-1}$. \end{theorem}

\begin{proof} 
We sketch the proof. Let $C$ be a positive number, to be chosen momentarily. Consider the function $R_C$ defined by

$$ R_C(z,t, {\overline z},{\overline t}) = Hr(z, t,{\overline 
z},{\overline t}) + 
C (||z||^2 - |t|^2)^m. \eqno (34) $$
Here $Hr$, the bihomogenization of $r$, has total degree $2m$. We may assume without loss of generality
that $m$ itself is even. Note that $R_C$ is bihomogeneous.
Suppose we can choose $C$ so that $R_C$ is positive on the unit sphere.
By Theorem 3.1 it follows that there is an integer $d$ and a holomorphic
polynomial mapping $g$ such that

$$ (||z||^2 + |t|^2)^d  R_C(z,t, {\overline z},{\overline t}) = ||g(z,t)||^2. \eqno (35) $$
Putting $t=1$ and then $||z||^2=1$ gives
$$ 2^d r(z,{\overline z}) = ||g(z,1)||^2$$
on the sphere, and hence yields the conclusion of the Theorem.

The intuition is simple. It suffices to show that $R_C$ is positive on $||z||^2 + |t|^2 = 2$. When $||z||^2=|t|^2=1$, we know that
$R_C$ is positive, because $r$ is positive on the sphere. By continuity, $R_C > 0$ when $| \ ||z||^2 - |t|^2 |$ is small. But, when this
quantity is large (at most $2$ of course),  the second term in (34) is large and positive. Since the first term achieves a minimum on a compact set,
we can choose $C$ large enough to guarantee that $R_C > 0$ away from the 
origin.
\end{proof}

The example $(|z_1|^2 - |z_2|^2)^2$ shows that non-negativity does not suffice for the conclusion.

Next we mention some related results concerning 
positive functions on the boundaries of strongly pseudoconvex domains.
L\o w [L] proved the following result. Suppose that $\Omega$ is a strongly pseudoconvex
domain with $C^2$ boundary, and $\phi$ is a positive continuous function on the boundary $b\Omega$.
Then there is a mapping $g$, holomorphic on $\Omega$, continuous on the closure of $\Omega$,
and taking values in a finite dimensional space, such that $\phi = ||g||^2$ on $b\Omega$.
Lempert ([L1], [L2]) considers strongly pseudoconvex domains with real-analytic boundary.
One of his results states that, given a positive continuous function $\phi$ on $b\Omega$,
there is a sequence $h_1,h_2,...$ of functions, holomorphic on $\Omega$, continuous on $b\Omega$, such that
$||h||^2 = \sum_j |h_j|^2$ converges on $b\Omega$ and agrees with $\phi$ there.
These theorems form part of work concerning embedding strongly pseudoconvex domains into balls.

Given L\o w's result, it is natural to ask whether Theorem 3.3 can be generalized. Recently Putinar and Scheiderer [PS]
provided an important example and a new technique concerning such generalizations. The author once asked the following question,
which, as Example shows 3.1 shows,  has a negative answer in general. Let $\Omega$ be a strongly pseudoconvex domain with an algebraic boundary.
Let $f(z,{\overline z})$ be a polynomial and assume that $f$ is positive on $b\Omega$. Is there a holomorphic polynomial mapping $g$, taking values
in a finite dimensional space, 
such that $f(z,{\overline z}) = ||g(z)||^2$ on $b\Omega$. The answer can be no! The following example also shows
that the holomorphic mapping $g$ constructed by L\o w does not extend holomorphically past the boundary, even when
the data $b\Omega$ and $\phi$ are algebraic. 

\begin{example} Put $r(z,{\overline z}) = |z_1(z_1^2-1)|^2 + |z_2|^2 - c^2$. Let $\Omega$ be the set of $z$ for which
$r(z,{\overline z}) < 0$. Put $f(z,{\overline z}) = m - |z_1|^2 |z_2|^2$.
For sufficiently small positive $c$, $\Omega$ is strongly pseudoconvex. For $M$ sufficiently large, $f>0$ on $b \Omega$.
But $f$ agrees with no squared norm on $b\Omega$. The idea of the proof, due to Putinar and Scheiderer, amounts to
considering the space ${\mathcal P}_2(b\Omega)$. Let $p=(1,c)$ and let $q=(-1,c)$. Simple calculation shows that
$$ r(p,{\overline p}) = r( q,{\overline q})= r(p,{\overline q}) = r(q,{\overline p}) = 0.  $$
If $f=||g||^2$ on $b\Omega$, then we would have the following:
$$ m - c^2 = f(p,{\overline p}) = ||g(p)||^2 $$
$$ m - c^2 = f(q,{\overline q}) = ||g(q)||^2 $$
$$ m + c^2 = f(p,{\overline q}) = f(q,{\overline p}) = \langle g(p), g(q) \rangle.$$
If these three conditions held, then the Cauchy-Schwarz inequality would imply the obviously false inequality
$$ -4 m c^2 = (m-c^2)^2 - (m+c^2)^2 \ge 0. $$
\end{example}

Dropping the term $|z_2|^2$ from the defining equation in Example 3.1 leads to an example of a domain in ${\bf C}$ where
the positivity property fails as well. The author believes that the original question should be rephrased along the following lines.

Let $X$ be an algebraic subset of ${\bf C}^n$.
One wishes to introduce the notation ${\mathcal P}_k(X)$ with the following meaning. Assume $X= \{u=0\}$, and let $z_1,...,z_k$ be points
such that $u(z_j, {\overline z_k}) = 0$ for all $j,k$. When $j=k$ we see that $z_j \in X$; for $j \ne k$ we see that
$z_k$ is in the Segre set determined by $z_j$. A Hermitian polynomial $f$ is in 
${\mathcal P}_k(X)$ if each matrix $f(z_j,{\overline z_k})$, formed by evaluation at such points, is nonnegative definite.
In Example 3.1, we see that the given $f$ is not in ${\mathcal P}_2(b\Omega)$.
This approach leads to a subtle difficulty: the number $k$ depends on the choice of defining equation $u$.
For example, the unit sphere can be defined for each positive integer $d$ by $u = ||z^{\otimes d}||^2 -1$. Each such function
is a unit times $||z||^2 -1$. After polarization, however, this property no longer holds. When $d=1$,
one cannot find distinct points $z_1$ and $z_2$ satisfying the above equations. For general $d$, however, one can find such sets
with $d$ distinct points. It therefore follows that one must define the appropriate notions for real polynomial ideals,
rather than for their zero sets. Doing so leads to a notion of {\it Hermitian complexity} for real polynomial ideals, introduced in [DP].

It is also natural to expect that a stability result holds; appropriate information on the ideal tells us how large $k$ needs to be.
For example, by Theorem 3.3, for the ideal $||z||^2 -1$  and for $f$ strictly positive, we need only consider $k=1$. For the ideal
generated by $r$ from Example 3.1, $k=1$ does not suffice.

\section{The one-dimensional case}

We return to our analysis of ${\mathcal Q}$ and ${\mathcal Q}'$.
By Lemma 2.1 we gain information about these sets by pulling back to one dimension.
Following but improving [D4] we completely analyze the one-dimensional case. Thus $n=1$ in this section.

First we introduce the reflection
of a Hermitian polynomial. This concept suggests that the Riemann sphere (rather than ${\bf C}$) is the right place to work.

\begin{definition} Let $r(z,{\overline z})$ be a Hermitian symmetric polynomial of degree $m$ in $z \in {\bf C}$.
We define a new Hermitian symmetric
polynomial $r^*$ called the {\it reflection} of $r$  by

$$ r^*(z,{\overline z}) = |z|^{2m} r({1 \over z}, {1 \over {\overline z}}). $$
\end{definition}

\begin{remark} The reflection is closely related to the bihomogenization:
$$ r^*(z,{\overline z}) = (Hr)(1,z, 1,{\overline z}). $$
This formula requires that $n=1$.  \end{remark}

Definition 4.1 is a bit subtle. For example, the reflection map is not injective, and the reflection
of a sum need not be the sum of the reflections. Reflection preserves neither degree in $z$ nor total degree.
Also, $r^{**}$ need not be $r$.

\begin{example} We compute three reflections: 
\begin{itemize}
\item Put $r(z,{\overline z}) = 1 + (z+{\overline z})^4 + |z|^2$. Then
$r^*(z,{\overline z}) =|z|^8 + (z + {\overline z})^4 + |z|^6$. 
\item Put $r(z,{\overline z}) = z^m + {\overline z}^m$. Then $r^* = r$. 
\item Put $r(z,{\overline z}) = |z|^{2k}$. Then, for all $k$,  $r^*(z,{\overline z}) = 1$. \end{itemize}
\end{example}

\begin{example} If $r(z, {\overline z}) = z^2 + {\overline z}^2$ and $s(z, {\overline z}) = z^3 + {\overline z}^3$,
then each is its own reflection by the previous example. But 
$$ (r+s)^*(z, {\overline z}) = |z|^6 \left( {1 \over z^2} + { 1 \over {\overline z}^2} + {1 \over z^3} + { 1 \over {\overline z}^3} \right)
= |z|^2 ( z^2 + {\overline z}^2) +  z^3 + {\overline z}^3 \ne r(z,{\overline z})^* + s(z, {\overline z})^*. $$ \end{example}

On the other hand we have the following useful statement, which we apply in the proof of Theorem 4.1. We will also apply the subsequent
lemma and its corollary

\begin{lemma} $r \in {\mathcal Q}$ if and only if $r^* \in {\mathcal Q}$. Also, $r \in {\mathcal Q}'$ if and only if $r^* \in {\mathcal Q}'$.\end{lemma}
\begin{proof} By the symmetry between $0$ and infinity, it suffices to prove one implication in each case. Suppose $r \in {\mathcal Q}$. Then
$Hr \in {\mathcal Q}$. Remark 4.1 impies that $r^* \in {\mathcal Q}$. The proof for ${\mathcal Q}'$ is essentially the same.
 \end{proof}

\begin{lemma} Let $r$ be a Hermitian symmetric polynomial in one variable. Assume $r(p,{\overline p})=0$.
If $r \in {\mathcal Q}'$, then $r$ is divisible (as a polynomial) by $|z-p|^2$. \end{lemma}
\begin{proof} If $sr= ||f||^2$, then the hypothesis implies $||f(p)||^2 = 0$.
 We see that $f_j(p)=0$ for all $j$,
and hence each component $f_j$ is divisible by $z-p$. We may cancel all factors of $|z-p|^2$ that divide $s$ from both sides of the equation.
Since $r$ is Hermitian and $r(p,{\overline p}) = 0$, both $(z-p)$ and its conjugate
divide $r$. Thus $r$ is divisible by $|z-p|^2$.\end{proof} 

\begin{corollary} Assume $r \in {\mathcal Q}'$ and $r$ contains pure terms ($z^k$ or ${\overline z}^k$).
Then $r(0,0) > 0$.  \end{corollary}
\begin{proof} If $r(0,0)= 0$, then the lemma implies $r$ is divisible by $|z|^2$ and hence has no pure terms. \end{proof}

\begin{example} Put $r(z,{\overline z}) = |z|^2 + (z +{\overline z})^4 + |z|^6$. Then $r$ is not in ${\mathcal Q}'$. \end{example}

In Example 4.3, $r$ has an isolated $0$ at $0$, but the $z^4$ term 
prevents $r$ from being in ${\mathcal Q}'$.
The problem is that the zero-set of $r$ is not properly defined.

Even in one dimension we must deal with the following point. There exist Hermitian symmetric polynomials whose
values are (strictly) positive, yet for which the infimum of the set of values is zero. Such polynomials
cannot be quotients of squared norms.

\begin{example} Put $f(x,y) = (xy-1)^2 + x^2$. For $x > 0$, we have $f(x,{1 \over x}) = x^2$ and hence $f$ achieves values arbitrarily
close to $0$. On the other hand $f(x,y)$ is evidently never $0$. Writing $f$ in terms
of $z$ and ${\overline z}$ gives a  Hermitian symmetric example.\end{example}

We next state and prove Theorem 4.1. It is particularly striking 
that the sets ${\mathcal Q}$ and ${\mathcal Q}'$
are the same. These sets are characterized by a simple condition on degree, which provides
the extra thing needed besides a properly defined zero-set.

\begin{theorem} Let $r \in {\mathcal P}_1(1)$ be a Hermitian symmetric non-negative polynomial in one variable. The following are equivalent:
\begin{itemize}
\item 1) There is a holomorphic polynomial $h$ such that $r = |h|^2 R$, $R>0$, and the total degree of $R$ is twice
the degree of $R$ in $z$.
\item 2) There is a holomorphic polynomial $h$ such that $r = |h|^2 R$ and $R \in {\mathcal Q}$.
\item 3) $r \in {\mathcal Q}$.
\item 4) There is a holomorphic polynomial $h$ such that $r = |h|^2 R$ and $R \in {\mathcal Q}'$.
\item 5) $r \in {\mathcal Q}'$.

\end{itemize} \end{theorem}
\begin{proof} First suppose that $r$ vanishes identically. If we take $h=0$ and $R=1$, then all the statements hold.
Henceforth we assume that $r$ does not vanish identically. Suppose 1) holds. First we note by the information on degree that
there is a unique term $c|z|^{2m}$ of highest degree, where $c > 0$. Set $\epsilon = {\rm inf}(R)$. We claim that $\epsilon > 0$.
Assuming this claim, consider the bihomogenization $HR$. For $t \ne 0$ we have 

$$ HR(z,t, {\overline z}, {\overline t}) = |t|^{2m} R({z \over t}, {{\overline z} \over {\overline t}}) \ge |t|^{2m} \epsilon. $$
For $t$ near $0$ however the values of $HR$ are near $c|z|^{2m}$. Hence, there is a positive constant $\delta$
such that 
$$ HR(z,t, {\overline z}, {\overline t}) \ge \delta (|z|^{2m} + |t|^{2m}). $$
Therefore $HR$ is strictly positive away from the origin in ${\bf C}^{n+1}$.
By Theorem 3.1, we conclude that $HR \in {\mathcal Q}(2)$. We recover $R$ from $HR$ by setting $t=1$. Hence $R \in {\mathcal Q}(1)$.
We obtain $r$ by multiplication by $|h|^2$; by Lemma 2.4, $r \in {\mathcal Q}(1)$. Since ${\mathcal Q} \subset {\mathcal Q'}$, we 
obtain 4) from 2) and 5) from 3). Thus, given the claim, 1) implies 2) implies 3) implies 5) and 2) implies 4).

It remains to prove the claim. By the assumption on degree of $R$,
there is a unique term $c|z|^{2m}$ of highest degree. Hence, for $|z|$ sufficiently large,
we have 
$$ R(z,{\overline z}) \ge {c \over 2} |z|^{2m}. \eqno (36) $$
Now suppose that ${\rm inf}(R) = 0$. We  can then find a sequence $z_{\nu}$
on which $R(z_\nu, {\overline z_\nu})$ tends to zero. Since $R$ is bounded below by a positive number on any compact set, we
may assume that $|z_\nu|$ tends to infinity. But setting $z=z_\nu$ violates (36).
We have now shown that 1) implies the rest of the statements. We finish by showing that 5) implies 1).

Assume that 5) holds; thus there is an $s$ for which $rs=||f||^2$. As 
above, if $r=0$ all the statements hold. Otherwise 
${\bf V}(r)$ must be the zero-set of a holomorphic polynomial $h$,
which  we may 
assume is  $h(z) = \prod (z-p_j)$.
Both sides of $rs=||f||^2$ are divisible by $|h|^2$. We put $R= {r \over |h|^2}$ and we see that $R \in {\mathcal Q}'$.
Furthermore $R > 0$. It remains to establish the information about its degree.
Suppose that the terms of degree $2m$ include a term of degree larger than $m$ in $z$. It follows that the reflected polynomial $R^*$
vanishes at the origin and yet contains pure terms. By Corollary 4.1, $R^*$ is not in ${\mathcal Q}'$, and by 
Lemma 4.1 $R$ is not in ${\mathcal Q}'$.  Hence no such term can exist.
Thus 5) implies 1). Hence all the statements are equivalent.
\end{proof}

\begin{corollary}${\mathcal Q}(1) = {\mathcal Q}'(1)$.  \end{corollary}

\begin{corollary} Let $r(z,{\overline z})$ be a Hermitian symmetric polynomial in one variable.
Then $r \in {\mathcal Q}$ if and only if the following holds:

Either $r$ vanishes identically, or the zero-set of $r$ is a finite set $\{p_1,...,p_K\}$ (repetitions allowed) such that 
$$ r(z,{\overline z}) = \prod_{j=1}^K |z-p_j|^2 s(z,{\overline z}), $$
and $s$ satisfies both of the following conditions:

1) $s$ is strictly positive.

2) The total degree of $s$ is twice the degree of $s$ in $z$. \end{corollary}

\begin{corollary} Suppose $r > 0$ but ${\rm inf}(r) = 0$. Then $r$ is not in ${\mathcal Q}'$. \end{corollary}
\begin{proof} This statement is a corollary of the proof of Theorem 4.1. \end{proof}

The general one-dimensional case relies on the (non-degenerate) bihomogeneous case in two dimensions.
After dividing out factors of the form $|h(z)|^2$, we reduce to the situation
where $Hr$ satisfies the hypotheses of Theorem 3.1. 

Consider any Hermitian symmetric polynomial with pure terms $2(z^k + {\overline z}^k)$.
We may write these terms as
$$ 2(z^k + {\overline z}^k) = |z^k+1|^2 - |z^k-1|^2 = |f_0|^2 - |g_0|^2. $$
Adding any $||f'||^2- ||g'||^2$ to the right hand side, where $f(0)=0$ and $g(0)=0$,
yields a function $r$ of the form

$$ r= ||f'||^2 + |f_0|^2 - ||g'||^2 - |g_0|^2 = ||f||^2 - ||g||^2. \eqno (37) $$
Putting $z=0$ in (37) shows that there is no constant $c$ such that $c<1$
and $||g||^2 \le c ||f||^2$. Hence the failure of such a constant to exist 
eliminates pure terms. In fact, the existence of such a constant 
gives Varolin's characterization (Theorem 5.1)
of ${\mathcal Q}(n)$  in all dimensions.

\subsection*{Pulling back to one dimension}

Assume we are in $n$ dimensions, where $n\ge 2$. We can combine Theorem 4.1 and Lemma 2.1
to give easily checkable necessary conditions for being in ${\mathcal Q}$ or ${\mathcal Q}'$.
First we note the following simple fact.

\begin{lemma} Suppose $r \in {\mathcal Q}$. Then ${\bf V}(r)$ is a complex algebraic variety.
Suppose $r \in {\mathcal Q}'$ and $r$ is not identically zero. Then ${\bf V}(r)$ is contained
in a complex algebraic variety of positive codimension. \end{lemma}

\begin{definition} A Hermitian symmetric polynomial
$r: {\bf C} \to {\bf R}$ satisfies property (W) if either $r$ is
identically $0$, or $r$ vanishes to finite even order $2m$ at $0$
and its initial form (terms of lowest total degree) is $c |t|^{2m}$. \end{definition}

In other words, the only term of lowest total degree is $c|t|^{2m}$.
For example, $ (2 {\rm Re} (t))^2 $ does not satisfy property (W). It
equals its initial form, which  is $t^2 + 2 |t|^2 + {\overline t}^2$. 

The following lemma from [D3] shows that property (W)
is necessary for being a quotient of squared norms. Note also
its relationship with Lemma 4.2 when $p=0$.

\begin{lemma} Suppose that $r \in {\mathcal Q}$.
Let $t \to z(t)$ be a polynomial mapping. 
Then the pullback function $t \to r(z(t), {\overline {z(t)}})$ satisfies  property (W). \end{lemma}

\begin{example} Non-negative bihomogeneous
polynomials not in ${\mathcal Q}'$:

$$ r(z, {\overline z}) = ( |z_1|^2 - |z_2|^2)^2 \eqno (38) $$

$$ h(z,{\overline z}) = ( |z_1 z_2|^2 - |z_3|^4)^2 + |z_1|^8
\eqno (39) $$

The zero-set of the polynomial $r$ from (38) is three real dimensions, and hence not contained in any complex
variety other than the whole space. Thus $r$ is not in ${\mathcal Q}'$ and hence not in ${\mathcal Q}$ either.
Alternatively, property (W) fails if we pullback using
$z(t)=(1+t,1)$, obtaining
$$ p(z(t), {\overline {z(t)}}) = (t + {\overline t} + |t|^2)^2. $$
This expression also violates the condition of Lemma 4.2.

The zero-set of $h$ is the complex variety defined by $z_1=z_3=0$. 
Yet $h$ is not in ${\mathcal Q}$. Put $z(t)= (t^2,1+t,t)$.
Then property (W) fails for the pullback. A simple computation shows that
the initial form of the pullback contains the term $t^4 {\overline t}^6$. \end{example}

These examples prove that the containment ${\mathcal Q} \subset {\mathcal P}_1$ is strict.

\section{Varolin's Theorems}

Varolin has extended Theorem 3.1 in two fundamental ways. The first way to extend the result
is to allow objects more general than polynomials. We may regard a homogeneous polynomial of degree $m$
on ${\bf C}^n$ as a section of the $m$-th power ${\mathcal H}^m$ of the hyperplane bundle over complex projective space
${\bf P}_{n-1}$. Hence a Hermitian symmetric polynomial

$$ \sum c_{\alpha \beta} z^\alpha {\overline z}^\beta $$
can be rewritten
$$ \sum c_{\alpha \beta} s_\alpha {\overline s_\beta}, \eqno (40) $$
where the $s_\alpha$ form a basis for sections of ${\mathcal H}^m$. If such a polynomial is non-negative,
then it can be regarded as a metric on the dual line bundle. 
Many of the ideas of this paper extend to metrics on holomorphic line bundles
over compact complex manifolds. We briefly discuss some of this material in Section 8.

Our main purpose in this current section is the other direction in which Varolin extended Theorem 3.1.
Suppose that $r$ is non-negative and bihomogeneous, but vanishes outside the origin.
We have seen that $r$ is sometimes in ${\mathcal Q}$ but other times it is not.
Varolin proved the following two results, by generalizing the proof of Theorem 3.1 and using a form of the resolution
of singularities, giving a complete solution to Analogue 1. The version of Theorem 5.1 differs in language from Theorem 1 as stated in [V],
but the two statements are easily seen to be equivalent. The version of Theorem 5.2 is essentially Proposition 4.2 in [V].

\begin{theorem} [Varolin] Suppose $r = ||f||^2 - ||g||^2$ is a bihomogeneous polynomial and the components of $f$ and $g$ are linearly independent.
Then $r \in {\mathcal Q}$ if and only if there is a $\lambda < 1$ such that
$||g||^2 \le \lambda ||f||^2$.
\end{theorem}

\begin{theorem} Suppose $r$ is as in Theorem 5.1. Then $r \in {\mathcal Q}$ if and only if
Property (W) holds for $z^*r$ for every rational map $z:{\bf C} \to {\bf C}^n$. \end{theorem}

By combining Theorem 5.2 with Theorem 4.1, we obtain a complete solution to Analogue 2.

\begin{theorem} For all $n$, ${\mathcal Q}(n) = {\mathcal Q}'(n)$. In other words, let $r$ be 
a Hermitian symmetric polynomial, not identically $0$. Then $r$ is a quotient of squared norms if and only if it divides a squared norm.\end{theorem}
\begin{proof} The containment ${\mathcal Q} \subset {\mathcal Q}'$ is trivial. Suppose that $r$ is not in ${\mathcal Q}(n)$.
Then the bihomogenization $Hr$ is not in ${\mathcal Q}(n+1)$.
By Theorem 5.2 there is a rational curve $z$ for which $z^*(Hr)$ violates property $W$. After clearing denominators
it follows that $z^*(Hr)$ is not in ${\mathcal Q}(1)$. By Theorem 4.1, $z^*(Hr)$ is not in ${\mathcal Q}'(1)$ either.
But then $Hr$ is not in ${\mathcal Q}'(n+1)$ and hence $r$ is not in ${\mathcal Q}'(n)$. Hence ${\mathcal Q}(n) = {\mathcal Q}'(n)$.
\end{proof}

It is possible to prove Theorem 5.3 by using Varolin's approach via the resolution of singularities.
The idea, roughly speaking, follows. Assume $rs=||f||^2$. Blow up the ideal of $||f||^2$ and cancel factors 
to reduce to the case where $r$ and $s$ are positive. By Theorem 3.1 each is an element of ${\mathcal Q}$.  The proof given here
is similar in spirit. Theorem 5.2 of Varolin enables the reduction to the one-dimensional case.
Theorem 4.1 of this paper handles the one-dimensional case, although the logic still passes through Theorem 3.1.

\section{Applications to Proper Mappings between Balls}

We first recall some facts about proper mappings between domains in 
complex Euclidean spaces. Let $\Omega$ and $\Omega'$ be bounded
domains in ${\bf C^n}$ and ${\bf C^N}$. A holomorphic mapping 
$f:  \Omega \to \Omega'$ is proper if $f^{-1}(K)$ is compact in  $\Omega$
whenever
$K$ is compact in $\Omega'$. When such an $f$ extends to be a continuous
mapping of the boundaries, it will be proper
precisely when it maps the boundary  $b\Omega$ to the boundary $b\Omega'$.

Let $f:B_n \to B_N$ be a proper holomorphic mapping.
When $f$ extends smoothly to the boundary we see that 
$ ||z||^2 =1$ implies $||f(z)||^2 = 1$, and hence there is an obvious connection
to squared norms. We will see more subtle relationships as well.

\medskip
We recall many facts about proper holomorphic mappings $f:B_n \to B_N$. See [F] and [D2]
for references.
\begin{itemize}

\item When $N < n$, every such $f$ is a constant. This conclusion follows
from the observation that positive dimensional
complex analytic subvarieties of the ball are noncompact.

\item When $n=N=1$, every such $f$ is a finite Blaschke product.
There are finitely many points $a_j$ in $B_1$, positive integer
multiplicities $m_j$, and an element $e^{i\theta}$ such that

$$ f(z) = e^{i\theta} \prod_{j=1}^K ({z-a_j \over 1-{\overline a_j}z}) ^{m_j}  \eqno
(41) $$
Note that (41) shows that there is no restriction on the denominator.
Every polynomial $q$ that is not zero on the closed ball
arises as the denominator of a rational function reduced to lowest terms.

\item When $n=N > 1$, a proper holomorphic
map $f:B_n \to B_n$ is necessarily an automorphism.
In particular $f$ is a linear fractional transformation with
denominator $ 1 - \langle z, a \rangle $ for $a \in B_n$.

\item There are proper holomorphic mappings $f:B_n \to 
B_{n+1}$ that are continuous but not smooth on the boundary sphere.

\item Assume $n\ge 2$.  If $f:B_n \to B_N$ is a proper map and has $N-n+1$
continuous derivatives on the boundary, then $f$ must be a rational
mapping [F].  Furthermore, by [CS], the denominator cannot vanish on the closed ball.
\end{itemize}

The author began his study of complex variables analogues of Hilbert's problem in order to verify the next result.
In it we want ${p \over q}$ to be in lowest terms, 
or else we have the trivial example
where $p(z) = q(z) (z_1, ..., z_n)$. 

\begin{theorem} Let $q:{\bf C^n} \to {\bf C}$ be a holomorphic polynomial,
and suppose that $q$ does not
vanish on the closed unit ball. 
Then there is an integer $N$ and a holomorphic polynomial
mapping $p : {\bf C^n} \to {\bf C^N}$ such that

1. ${p \over q}$ is a rational proper mapping between $B_n$ and $B_N$. 

2. ${p \over q}$ is reduced to lowest terms.
\end{theorem}

\begin{proof} The result is trivial when $q$ is a constant and it is easy when $n=1$. 
When the degree $d$ of $q$
is positive in one dimension, we define $p$ by $p(z) = z^d {\overline q}({1 \over
z})$. Such a proof cannot work in higher dimensions.
The minimum integer $N$ can be arbitrarily large
even when $n=2$ and the degree of $q$ is also two.

Now assume $n\ge 2$.
Suppose that $q(z) \ne 0$ on the closed ball.
Let $g$ be an arbitrary polynomial such that $q$ and $g$ have no common factor.
Then there is a constant $c$ so that 
$$ |q(z)|^2 - |cg(z)|^2 > 0 \eqno (42) $$
for $||z||^2 = 1$. We set $p_1 = cg$.

By Theorem 3.3, $|q|^2 - |p_1|^2$ agrees on the sphere with a squared norm of a 
holomorphic
polynomial mapping. Thus there are polynomials
$p_2, ...,p_N$ such that

$$ |q(z)|^2 - |p_1(z)|^2 = \sum_{j=2}^N |p_j(z)|^2 \eqno (43) $$
on the sphere.
It then follows that ${p \over q}$ does the job. \end{proof}

Theorem 3.3 can be used also to show 
that one can choose various components $p$ of a proper
holomorphic polynomial mapping arbitrarily, assuming only that they satisfy 
the necessary condition $||p(z)||^2 < 1$ on the sphere.

\begin{corollary} Let $p:{\bf C}^n \to {\bf C}^k$ be a polynomial with
$||p(z)||^2 < 1$ on the unit sphere. Then 
there is a polynomial mapping $g$
such that the polynomial $p \oplus g$ is a proper holomorphic mapping between balls. \end{corollary}

\begin{proof}
Note that $1 - ||p(z)||^2$ is a polynomial
that is positive on the sphere. Hence we can find a holomorphic polynomial mapping $g$
such that 
$$ 1 - ||p(z)||^2 = ||g(z)||^2 $$
on the sphere. We may assume that not both $p$ and $g$ are constant.
Then 
$p\oplus g$ is a non-constant holomorphic polynomial mapping 
whose squared norm equals unity on the sphere.
By the maximum principle $p\oplus g$ is the required mapping.
\end{proof}

\section{roots of squared norms}

We check a simple fact noted in the introduction. If $r \in {\rm rad}({\mathcal P}_\infty)$, then
we have $r^{N-1} r \in {\mathcal P}_\infty$. Thus there exists a $q$ for which $qr \in {\mathcal P}_\infty$.
Hence $r \in {\mathcal Q}'$. Thus 
$$ {\rm rad} ({\mathcal P}_\infty) \subset {\mathcal Q}'. $$

In this section we provide additional information about ${\rm rad} ({\mathcal P_\infty})$.
Perhaps the most striking statement is its relationship with ${\mathcal P}_2$. We may regard ${\mathcal P}_2$ as a closed cone
in real Euclidean space. We have the following result, in which ${\rm int}$ denotes interior.

\begin{theorem} The following containments hold, and all are strict: 
$$ {\rm int}({\mathcal P}_2) \subset {\rm rad} ({\mathcal P_\infty}) \subset {\mathcal P}_2 \subset {\mathcal L} \subset {\mathcal P}_1. \eqno (44) $$
 \end{theorem}

\begin{corollary} Suppose $r \in {\rm rad} ({\mathcal P}_\infty)$. Then $r \in {\mathcal P}_2$ and ${\rm log} (r)$ is plurisubharmonic.
\end{corollary}

We discuss but do not prove the first containment. We begin with a surprising fact and continue by establishing the other containments.

\begin{example} ${\rm rad} ({\mathcal P_\infty})$ is not closed under sum.
Choose an $R \in {\rm rad} ({\mathcal P_\infty})$ of the form $r=||f||^2 - |g|^2$, where $f$ and $g$
are homogeneous of degree $m$ in the variables $z_2$ and $z_3$ and 
their components are linearly independent. Let $r = |z_1|^{2m}$.
Then $R+r$ is not in ${\rm rad} ({\mathcal P_\infty})$; for each $N$ the function $(R+r)^N$ will contain the
term $-N|g|^2 |z_1|^{2m(N-1)}$. This term arises nowhere else in the expansion, and hence $(R+r)^N$ is not a squared norm.
If we can find such an $R$ in ${\rm rad}({\mathcal P_\infty})$, we have an example. By Example 2.1,
$R = (|z_2|^2 + |z_3|^2)^4 - \beta |z_2 z_3|^4$ works if $6 < \beta < 8$.
 \end{example}

The function $R=R_\beta$ from this example shows that ${\rm rad} ({\mathcal P_\infty})$ is not closed under limits.
It is easy to see that ${\rm rad} ({\mathcal P_\infty})$ is closed under products.
We next establish most of the containments from Theorem 7.1.

\begin{lemma} ${\rm rad}({\mathcal P}_\infty) \subset {\mathcal P}_2$. 
\end{lemma}

\begin{proof} Suppose that $r^N = ||f||^2$.
By the usual Cauchy-Schwarz inequality,
$$ (r(z,{\overline z})r(w,{\overline w}))^N  = ||f(z)||^2 ||f(w)||^2 \ge
|\langle f(z), f(w) \rangle|^2 = |r(z,{\overline w})|^{2N}. \eqno (45) $$
Since $N>0$, we may take $N$-th roots of both sides of (45) and preserve 
the direction of the inequality to obtain (13). By the principal minors test for 
non-negative definiteness, we see that $r\in {\mathcal P}_2$. \end{proof}

\begin{remark} Suppose $r$ satisfies (13). If $r(z,{\overline z}) > 0$
for a single point $z$, then
 $r \in {\mathcal P}_2 \subset {\mathcal P}_1$. Positivity at one point is required.
When $r$ is minus a squared norm, and not identically
zero, (13) holds and $r$ is not in ${\mathcal P}_1$.
If $r(z,{\overline z})$ is positive at one point, then
(13) is equivalent to
$$ r(w,{\overline w}) \ge { |r(z,{\overline w})|^2 \over r(z, {\overline z}) } $$
and hence $r(w, {\overline w})$ is non-negative for all $w$. \end{remark}

\begin{lemma} ${\mathcal P}_2 \subset {\mathcal L}$. \end{lemma}
\begin{proof} If $r \in {\mathcal P}_2$, then (13) holds.  
Since equality holds when $z=w$, the right-hand side of (13)
has a minimum at $z=w$, and hence its complex Hessian there
 is non-negative definite. Computing the Hessian
shows that the matrix with $i,j$ entry equal to 
$$ r \ r_{z_i {\overline z_j}} - r_{z_i} r_{{\overline z_j}} $$
is non-negative definite. Computing the Hessian of ${\rm log}(r)$ leads to 
the same condition.\end{proof}

We continue to develop a feeling for the Cauchy-Schwarz inequality. 
By (3) there are holomorphic polynomial mappings $f$ and $g$, 
taking values in finite-dimensional
spaces, such that 

$$ r(z,{\overline z}) = ||f(z)||^2 - ||g(z)||^2. $$
We may assume that there are no
linear dependence relations among the components of $f$ and $g$,
but even then the representation is not unique.

In the next Proposition we allow $f$ and $g$ to be Hilbert space valued holomorphic mappings.
In the polynomial case the Hilbert space is finite-dimensional.

\begin{proposition} Suppose that ${\mathcal H}$ is a Hilbert space, and that
$f$ and $g$ are holomorphic mappings to ${\mathcal H}$.
Put 
$$ r(z,{\overline w}) = \langle f(z), f(w) \rangle - \langle g(z), g(w) \rangle \eqno
(46) $$
Then (13) holds if and only if, for every pair of points $z$ and $w$, we have

$$ || f(z) \otimes g(w) - f(w) \otimes g(z) ||^2 $$
$$\le ||f(z)||^2 ||f(w)||^2 - |\langle f(z),f(w) \rangle |^2
+  ||g(z)||^2 ||g(w)||^2 - |\langle g(z),g(w) \rangle |^2  \eqno (47) $$ \end{proposition}

\begin{proof} Begin by using (46) to express (13) in terms of $f$ and $g$.
The resulting inequality is then seen to be equivalent to (47).
To see this, expand the squared norm on the left side
of (47), and use the identity 
$ \langle u_1 \otimes v_1, u_2 \otimes v_2 \rangle = 
\langle u_1, u_2 \rangle \langle v_1, v_2 \rangle $. \end{proof}

\medskip
Consider the right side of (47); the two terms involving $f$, and the two terms
involving $g$, are each non-negative by the usual
Cauchy-Schwarz inequality. Their sum is thus non-negative. Version (47) of the Cauchy-Schwarz
inequality demands more; their sum must 
bound an obviously non-negative expression that reveals the symmetry of
the situation. The left side of (47) has the interpretation
as $||(f \wedge g)(z,w)||^2$, but the author does not know how to use this
information.

It remains to discuss the first containment in (44). It follows from results in [CR3] and [V].
These results, which are expressed in terms of metrics on line bundles, involve strict forms
of (13). They imply, when $r$ is bihomogeneous and satisfies a strict form of (13),
that $R \in {\rm rad} ({\mathcal P_\infty})$. The strict forms of (13) are open conditions on the coefficients,
and hence we obtain the first containment in (44).

\section{Isometric imbedding for holomorphic bundles}

Let $r$ be a bihomogeneous polynomial
that is positive away from the origin in ${\bf C^{N+1}}$.
The link to bundles arises by first
considering complex projective space ${\bf P_N}$, the collection
of lines through the origin in
${\bf C^{N+1}}$. We have the usual open covering
given by open sets $U_j$ where $z_j \ne 0$. In $U_j$ we define
$f_j$ by

$$ f_j(z,{\overline z}) = { r(z, {\overline z}) \over |z_j|^{2m} }.  \eqno (48) $$
On the overlap $U_j \cap U_k$ these functions then transform via

$$ f_k = |({z_j \over z_k})^m|^2 f_j . \eqno (49) $$
Since $({z_j \over z_k})^m$ are the transition functions for the $m$-th
power of the universal line bundle ${\bf U}^m$, the functions $f_j$ determine
a Hermitian
metric on ${\bf U}^m$.

We will reformulate
Theorem 3.1 in this language and then generalize it.

Let $r$ be a bihomogeneous polynomial of degree $2m$.
It defines via (48) a metric on ${\bf U}^m$
if and only if it is positive as a function away from the origin.
This metric is already a pullback
of the Euclidean metric if and only if $r \in {\mathcal Q}$. 
Some tensor power of the bundle with itself is a pullback
if and only if $r \in {\rm rad}({\mathcal P}_\infty)$. If $r \in {\mathcal P}_2$, then $r \in {\mathcal L}$. 
This condition is equivalent
to the negativity of the curvature of the bundle, or to the pseudoconvexity
of the unit ball in the total space of the bundle.

The previous paragraph applies in particular to
the function $r_\lambda$ from Example 2.1.
When $ \lambda < 16$, $r_\lambda$ 
is strictly positive away from the origin,
and hence defines a metric on ${\bf U}^4$ over ${\bf P_1}$.
By varying the parameter $\lambda$ we see that the various positivity properties
of bundle metrics are also distinct.

We next restate Theorem 3.1.

\begin{theorem} Let $({\bf U}^m, r)$ denote 
the $m$-th power of the universal line bundle over
${\bf P_n}$ with metric defined by $r$.
Then there are integers $N$ and  $d$ such that $({\bf U}^{m+d}, ||z||^{2d} r(z, {\overline z}))$
is a (holomorphic) pullback $g^*({\bf U}, ||L(\zeta)||^2)$
of the standard metric on the universal bundle over  ${\bf P_N}$. 
The mapping $g: {\bf P_n} \to {\bf P_N}$ is a holomorphic (polynomial) embedding and
$L$ is an invertible linear mapping.

$$ ({\bf U}^m, r) \otimes ({\bf U}^d, ||z||^{2d}) = 
({\bf U}^{m+d}, ||z||^{2d} r(z,{\overline z})) = ({\bf U}^{m+d}, ||g(z)||^2) $$
We have the bundles and metrics $$\pi_1 : ({\bf U}^m,r) \to {\bf P_n}$$
$$\pi_2 : ({\bf U}^{m+d},||z||^{2d} r) \to {\bf P_n}$$
$$\pi_3 : ({\bf U},||L(\zeta)||^2) \to {\bf P_N}$$

Thus $\pi_1$ is not an isometric pullback of $\pi_3$, but, for sufficiently large $d$,
$\pi_2$ {\it is} such a pullback. \end{theorem}

This formulation suggests a generalization
to more general Hermitian bundles.
See [CD3] for the precise definitions
of globalizable metric and the needed sharp form of inequality (13). See [V] for an improved exposition that allows
for degenerate metrics.
A version of (13) arises also in Calabi's work [Cal] on isometric
imbeddings of the tangent bundle.
The main result of [CD3] is the following isometric imbedding theorem
for holomorphic bundles.

\begin{theorem} Let $M$ be a compact complex manifold. 
Let $E$ be a vector bundle of rank $p$ over $M$ with globalizable
Hermitian metric $G$. Let $L$ be a line bundle over $M$ with
globalizable Hermitian metric $R$, and suppose that $L$ is negative and
that $R$ satisfies a sharp form of (13). Then there is an integer $d_0$ such that, for
all $d$ with $d \ge d_0$, there is a holomorphic imbedding $h_d$ with
$h_d : M \to {\bf G}_{p,N}$ such that
$ E \otimes L^d = h_d^* {\bf U}_{p,N} $, and $ G R^d = h_d ^* (g_0) $. \end{theorem}

The special case where the base manifold is  complex projective space
${\bf P}_{n-1}$ gives us Theorem 3.3. We let $E$ be a power ${\bf U}^m$ 
of the universal bundle, with metric determined by the
bihomogeneous polynomial $r$, and we let $L$ be the universal bundle
${\bf U}$ with the Euclidean metric.  A matrix analogue
of Theorem 3.1 holds, where $E$ is the bundle of rank $k$ given by
$k$ copies of ${\bf U}^m$. See [D3] for details.

\begin{corollary} Let $M(z, {\overline z})$ be
a matrix of bihomogeneous polynomials of the same degree that
is positive-definite away from the origin. 
Suppose $R$ is a bihomogeneous polynomial that is
positive away from the origin and satisfies a sharp form of (13).
Assume also that $\{ R <1 \}$ is a strongly pseudoconvex
domain. 
Then there is an integer $d$ and a matrix $A$ of holomorphic
homogeneous polynomials such that 
$$ R(z,{\overline z})^d M(z, {\overline z}) = A(z)^* A(z). \eqno (50) $$ 
In particular we can choose $R(z,{\overline z}) = ||z||^{2d}$. \end{corollary}

A matrix $C$ is positive definite
if and only if its components $c_{ij}$ can be expressed as
inner products $\langle e_i, e_j \rangle$ of basis elements;
we can thus factor a positive definite matrix of constants
as $C= A^*A$. When the entries depend real-analytically on parameters
it is generally impossible to make $A$ depend holomorphically on these
parameters. Writing an operator-valued real-analytic function as in the right side
of (50) is called {\it holomorphic factorization}.
See [RR] for classical results about holomorphic factorization
of operator-valued holomorphic functions of one complex variable.
In the situation of Corollary 9.1, 
one cannot factor $M$ holomorphically, 
but one can factor
$R^d M$ holomorphically when $d$ is sufficiently large.

\section{signature pairs}

Let $r$ be a Hermitian symmetric polynomial with ${\bf s}(r) = (A,B)$. The underlying matrix of coefficients of $r$
is diagonal if and only if we can write 
$$ r = \sum_\alpha  c_\alpha |z|^{2 \alpha} = \sum_{\alpha} c_\alpha |z_1|^{2 \alpha_1} |z_2|^{2 \alpha_2} ... |z_n|^{2 \alpha_n}. $$
Define the moment map ${\bf m}$ by
$$ z \to x = (x_1,...,x_n) = (|z_1|^2,...,|z_n|^2) = {\bf m}(z). \eqno (51) $$
Thus in the diagonal case there is a (real) polynomial $R$ in $x$ such that
$$ r(z,{\overline z}) = R({\bf m}(z)) = R(x). \eqno (52) $$
Then $R$ has $A$ positive coefficients and $B$ negative coefficients.

The relationship between the diagonal case
and the general case parallels the relationship between Theorem 3.2 and Theorem 3.1. 
We show next that the special diagonal situation suffices for finding examples of maximal collapsing of rank.

\begin{lemma} Assume $m \ge 2$. For $t \in {\bf R}$ put $p(t) = t^{2^m} + 1$. Then there is a polynomial $q(t)$ such that
\begin{itemize}
\item All $2^{m-1}+ 1$ coefficients of $q$ are positive.

\item $p(t) = q(t) q(-t)$
\end{itemize} \end{lemma}
\begin{proof} Regard $t$ as a complex variable. Put $\omega = e^{ \pi i \over 2^m}$.
The roots of $p$ occur when $t$ is a $2^m$-th root of $-1$, and hence are odd powers of $\omega$. 
Factor $p$ into linear factors:
$$ p(t) = \prod_j (t-\omega^{2j+1}). $$
The roots are symmetrically located in the four quadrants. 
We define $q$ by taking the product over the terms where ${\rm Re}(\omega^{2j+1}) < 0$.
Each such factor has a corresponding conjugate factor. Hence
$$ q(t) = \prod (t^2 - 2 {\rm Re}(\omega^{2j+1})t + 1), $$
and all the coefficients of $q$ are positive. The remaining terms in the factorization of $p$ define $q(-t)$, 
and the result follows. \end{proof}

We illustrate Lemma 9.1 with an example. Set $a=\sqrt{ 4 \pm 2 \sqrt{2}}$. 
Then
$$ t^8 + 1 = (t^4 + a t^3 + {a^2 \over 2} t^2 +  a t + 1) (t^4 - a t^3 + {a^2 \over 2} t^2 - a t + 1). \eqno (53) $$

Bihomogenization leads to the following result from [DL].

\begin{proposition} There are Hermitian symmetric polynomials
$q$ and $r$ such that the following hold:
\begin{itemize} 
\item $q$ and $r$ are each bihomogeneous of total degree $2^m$. 

\item ${\bf s}(q) = (2^{m-1} + 1, 0)$.

\item ${\bf s}(r) = (2^{m-2} +1, 2^{m-2})$.

\item ${\bf s}(qr) = (2, 0)$.
\end{itemize}
\end{proposition}

\begin{corollary} For each integer $k$ of the form $2^{m-1} +1 $, there exists $r \in {\mathcal Q}$ of rank $k$
such that $||g||^2 r = ||f||^2$ and $||f||^2$ has rank $2$. \end{corollary}

Consider this proposition and corollary from the point of view of starting with $r$. It is a non-negative Hermitian polynomial
with signature pair $(2^{m-2}+1, 2^{m-2})$ and rank $2^{m-1} + 1$. 
By Proposition 9.1, $r$ is a quotient of squared norms, where
the rank of the numerator is $2$. For example, (53)
provides an example of a polynomial $r \in {\mathcal Q}$ whose signature pair of $r$ is $(3,2)$.
The rank of the numerator is $2$ and the rank of the denominator is $5$. The point of the Corollary is that by
choosing larger values of $m$, we can make ${\bf r}(qr)=2$, while
the ranks of the factors are arbitrarily large. This phenomenon illustrates the same warning required
in our discussion near (9.1) of Pfister's theorem in the real case. 

The next result shows that we cannot decrease the rank to $1$. On the other hand,
its conclusion is false for real-analytic Hermitian symmetric functions. Consider the identity $1 = e^{||z||^2} e^{-||z||^2}$.
If we expand the exponential as a series, then the signature pairs of the factors would be $(\infty,0)$ and $(\infty, \infty)$. Yet their product
has signature pair $(1,0)$. We return to the polynomial case.

\begin{proposition} ([DL]) Let $p$ and $q$ be Hermitian symmetric polynomials with ${\bf r}(pq)=1$.
Then ${\bf r}(p) = {\bf r}(q) = 1$. \end{proposition}

Examples from [DL] show that it is difficult to determine precisely
what happens to the rank of a Hermitian symmetric polynomial $q$ under multiplication.

The crucial information in the statement of the next proposition
is that the integers are non-zero.  We have seen 
already that we can obtain $(2,0)$ for the signature pair
of a product when one of the factors has signature pair $(A,0)$.
What happens if we insist that neither factor has signature pair $(k,0)$,
in other words, that neither factor is a squared norm? Remarkably, we can still get $(2,0)$.
In fact we can get any pair except $(0,0)$ (obviously), $(1,0)$, or $(0,1)$.

\begin{proposition} [DL] Assume $N \ge 2$. Then there exist Hermitian 
symmetric polynomials $r_1$ and $r_2$
such that ${\bf s}(r_j) = (A_j,B_j)$, none of the four integers $A_j$ or $B_j$ is zero,
and such that ${\bf s}(r_1 r_2) = (N,0)$. \end{proposition}

In other words, given an integer $N$ at least $2$, we can find a squared norm with rank $N$
which can be factored such that neither factor is a squared norm.

\section{bibliography}

[AM] Jim Agler and John E. McCarthy, Pick interpolation and Hilbert function spaces. 
Graduate Studies in Mathematics 44, American Mathematical Society, Providence, RI, 2002.

[Cal] E. Calabi, Isometric imbedding of complex manifolds, 
Annals of Math 58 (1953), 1-23.

[CD1] David W. Catlin and John P. D'Angelo,
A stabilization theorem for Hermitian forms and 
applications to holomorphic mappings, Math Research Letters 3 (1996), 149-166. 

[CD2] David W. Catlin and John P. D'Angelo, 
Positivity conditions for bihomogeneous polynomials, Math Research Letters 4 (1997), 1-13.

[CD3] David W. Catlin and John P. D'Angelo,
An isometric imbedding theorem for holomorphic bundles, Math Research Letters 6 (1999), 1-18.

[CS] J. A. Cima and T. J. Suffridge, Boundary behavior of rational proper maps, Duke
Math J. 35 (1988), 83-90.

[D1] John P. D'Angelo, Several Complex Variables and the Geometry of Real Hypersurfaces,
CRC Press, Boca Raton, 1993.

[D2] John P. D'Angelo, Proper holomorphic mappings, 
positivity conditions, and isometric imbedding, J. Korean Math Society, May 2003, 1-30.

[D3] John P. D'Angelo, Inequalities from Complex Analysis, Carus Mathematical Monograph No. 28, 
Mathematics Association of America, 2002.

[D4] John P. D'Angelo, An analogue of Hilbert's seventeenth problem in one complex dimension, Contemp. Math. 395 (2006), 47-58. 

[D5] John P. D'Angelo, Complex variables analogues of Hilbert's seventeenth problem,  Internat. J. Math.  16  (2005), 609-627.

[DL] John P. D'Angelo and Jiri Lebl, Hermitian symmetric polynomials and CR complexity, to appear in Journal Geometric Analysis.

[DP] John P. D'Angelo and Mihai Putinar, Hermitian complexity of real polynomial ideals, preprint 2010.

[DV] John P. D'Angelo and Dror Varolin, Positivity conditions for Hermitian symmetric functions, Asian J. Math. 7 (2003), 1-18. 

[F1] Franc Forstneric, Extending proper holomorphic maps of positive codimension, 
Inventiones Math., 95(1989), 31-62.

[G] Dusty Grundmeier, Signature Pairs for Group-Invariant Hermitian Polynomials , to appear in Intl. J. Math.

[H] David E. Handelman, Positive Polynomials, Convex Integral Polytopes, 
and a Random Walk Problem, Lecture Notes in
Mathematics 1282, Springer-Verlag, Berlin.

[HLP] G. H. Hardy, J. E. Littlewood, and G. Polya,
Inequalities, Cambridge, At the University Press, 1934.

[He] J. William Helton, ``Positive'' noncommutative polynomials are sums of squares,
Ann. of Math. (2) 156 (2002), 675-694.

[HP] J. W. Helton and M. Putinar, Positive polynomials in scalar and matrix variables, the spectral theorem, and optimization,
Operator Theorey, Structured Matrices, and Dilations, (2207), 101-176.

[L1] L. Lempert, Imbedding Cauchy-Riemann manifolds 
into a sphere, International Journal of Math 1 (1990), 91-108.

[L2] L. Lempert, Imbedding pseudoconvex domains into a ball, 
American Journal of Math 104 (1982), 901-904.

[L] E. L\o w, Embeddings and proper holomorphic maps of 
strictly pseudoconvex domains into polydiscs and balls,
Math Z. 190 (1985), 401-410.

[PD] Alexander Prestel and Charles N. Delzell, Positive Polynomials: From Hilbert's
17th Problem to Real Algebra, Springer-Verlag, Berlin, 2001.

[PF] A. Pfister, Zur Darstellung deﬁniter Funktionen als Summe von Quadraten. Invent. Math. 4 (1967), 229–237.

[PS] Mihai Putinar and Claus Scheiderer, Sums of Hermitian squares on pseudoconvex boundaries, to appear in Math Research Letters.

[Q] Daniel G. Quillen, On the Representation of Hermitian Forms
as Sums of Squares, Inventiones math. 5 (1968), 237-242.

[R1] B. Reznick, Some concrete aspects of Hilbert's 17th Problem, Contemp. Math. 253 (200), 251-272. 

[R2] B. Reznick, Uniform denominators in Hilbert's seventeenth problem, Math. Z. 220 (1995) 75-97.

[RR] M. Rosenblum and J. Rovnyak, The factorization problem for
non-negative operator valued functions, Bulletin A.M.S. 77 (1971),
287-318.

[S] Claus Scheiderer, Positivity and sums of squares: a guide to recent results,
pages 271-324 in Emerging applications of algebraic geometry, IMA Vol. Math. Appl., 149, Springer, New York, 2009. 

[Str] Emil Straube, Lectures on the $L^2$- Sobolev Theory of the ${\overline \partial}$-Neumann Problem, ESI Lectures in Mathematics and Physics, European Math. Society, 2010.

[TY] Wing-Keung To and  Sai-Kee Yeung,
Effective isometric embeddings for certain Hermitian holomorphic line bundles,  J. London Math. Soc. (2)  73  (2006),  no. 3, 607-624.

[V] Dror Varolin, Geometry of Hermitian algebraic functions: Quotients of squared norms,  American Journal of Math  130  (2008), 291-315.

\end{document}